\documentclass[times,doublespace]{grp_q_interpolation_rg1_arxiv}

\usepackage[all]{xy}
\usepackage{hyperref}
\hypersetup{colorlinks,linkcolor=red,citecolor=blue,urlcolor=black}

\newcommand{\ZZ}{\mathbb{Z}}
\newcommand{\CC}{\mathbb{C}}
\newcommand{\NN}{\mathbb{N}}  
\newcommand{\Glie}{\mathfrak{g}}
\newcommand{\Aq}{\mathcal{A}}

\newcommand{\Uc}{U(\mathfrak{sl}_2)}
\newcommand{\Uh}{U_h(\mathfrak{sl}_2)}
\newcommand{\Urh}[1]{U_{#1}(\mathfrak{sl}_2)}
\newcommand{\Uhh}{U_{h,h'}(\mathfrak{sl}_2,g)}

\newcommand{\cmod}{\mathcal{C}(\mathfrak{sl}_2)}
\newcommand{\cmodhh}{\mathcal{C}^{\: \! h,h'}(\mathfrak{sl}_2,g)}
\newcommand{\cmodh}{\mathcal{C}^{\: \! h}(\mathfrak{sl}_2)}

\newcommand{\grot}{\text{Rep} (\mathfrak{sl}_2)}
\newcommand{\groth}{\text{Rep}^{h} (\mathfrak{sl}_2)}

\newcommand{\grothh}{\text{Rep}^{h,h'} (\mathfrak{sl}_2,g)}

\newcommand{\xxn}[4]{\left( Q_{#4} \, T^{\: \! {\{ H_{#1}^{#2} #3 \}}_{\! \! \: Q_{#4}}} \right)^{H_{#1}^{#2} #3} - \ \left( Q_{#4} \, T^{\: \! {\{ H_{#1}^{#2} #3 \}}_{\! \! \: Q_{#4}}} \right)^{-H_{#1}^{#2}}}
\newcommand{\xxd}[4]{Q_{#4} \, T^{\: \! {\{ H_{#1}^{#2} #3 \}}_{\! \! \: Q_{#4}}} \ - \ Q_{#4}^{-1} \, T^{\: \! - {\{ H_{#1}^{#2} #3 \}}_{\! \! \: Q_{#4}}}}
\newcommand{\xxf}[4]{\frac{\xxn{#1}{#2}{#3}{#4}}{\xxd{#1}{#2}{#3}{#4}}}

\newcommand{\xxp}[2]{\prod_{e = 1, -1} \xxf{e}{#1}{#2}{}}

\begin{document}
\parindent = 0pt

\title{Groupes Quantiques d'Interpolation de Langlands de Rang 1}
\shorttitle{Groupes Q. d'Interpolation de Langlands de Rang 1}

\author{Alexandre Bouayad}
\abbrevauthor{A. Bouayad}
\headabbrevauthor{A. Bouayad}

\address{Université Paris 7, Institut de Mathématiques de Jussieu, 175, rue du Chevaleret 75013, \\ Paris - France.}

\correspdetails{\href{mailto:bouayad@math.jussieu.fr}{bouayad@math.jussieu.fr}}

\received{July 2011}

\begin{abstract}
On étudie une certaine famille, paramétrée par un entier $g \in \NN_{\geq 1}$, de doubles déformations $\Uhh$ de l'algèbre enveloppante $\Uc$, dans l'esprit de \cite{hernandez}. On prouve que $\Uhh$ déforme simultanément les groupes quantiques $\Uh$ et $\Urh{\! gh'}$. On montre que cette propriété d'interpolation explique la dualité de Langlands pour les représentations des groupes quantiques en rang 1. On résout ainsi une conjecture de \cite{hernandez} dans ce cas : on prouve pour tout $g \in \NN_{\geq 1}$ l'existence de représentations de $\Uhh$ qui déforment simultanément deux représentations Langlands duales. On étudie aussi plus généralement la théorie des représentations de rang fini de $\Uhh$.
\end{abstract}

\maketitle

\section*{Interpolating Langlands Quantum Groups of Rank 1}
\tabsize We study a certain family, parameterized by an integer $g \in \NN_{\geq 1}$, of double deformations $\Uhh$ of the envelopping algebra $\Uc$, in the spirit of \cite{hernandez}. We prove that $\Uhh$ simultaneously deforms the quantum groups $\Uh$ and $\Urh{\! gh'}$. We show this interpolating property explains the Langlands duality for the representations of the quantum groups in rank 1. Hence we prove a conjecture of \cite{hernandez} in this case : we prove for all $g \in \NN_{\geq 1}$ the existence of representations of $\Uhh$ which simultaneously deform two Langlands dual representations. We also study more generaly the finite rank representation theory of $\Uhh$.

\newpage
\section{Introduction}
Soit $\Glie$ une algèbre de Kac-Moody symétrisable et ${}^L \Glie$ sa duale de Langlands, dont la matrice de Cartan est la transposée de celle de $\Glie$. Il existe une dualité entre les représentations de $\Glie$ et ${}^L \Glie$ : voir \cite{hernandez2, hernandez, kashiwara, littelmann, mcgerty}. \\

La dualité de Langlands pour les représentations peut être décrite au niveau des caractères. Supposons pour simplifier que $\Glie$ soit de type fini et simple (${}^L \Glie$ est alors aussi de type fini et simple). Soit $I$ l'ensemble des sommets du diagramme de Dynkin de $\Glie$ et $r_i$ ($i \in I$) les longueurs des racines associées. On note $r$ la longueur maximale parmi les longueurs $r_i$. La longueur $r$ est égale à $1$ lorsque $\Glie$ est simplement lacée, à $2$ pour $B_l$, $C_l$ et $F_4$, et à $3$ pour $G_2$. Supposons que le plus haut poids $\lambda$ de la représentation irréductible de dimension finie $L(\lambda)$ s'écrive sous la forme
$$ \lambda \ = \ \sum_{i \in I} (1-r + r_i) \: \! m_i \: \! \omega_i \, , \quad m_i \in \NN \, , $$
où $\omega_i$ sont les poids fondamentaux de $\Glie$. En d'autres mots, notant $P$ le réseau des poids de $\Glie$, supposons que $\lambda$ soit dominant et appartienne au sous-réseau $P' \subset P$ engendré par $(1 - r + r_i) \: \! \omega_i$, $i \in I$. Le caractère de $L(\lambda)$ s'écrit
$$ \chi (L(\lambda)) \ = \ \sum_{\nu \in P} d(\lambda, \nu) \: \! e^{\nu} \, , \quad d(\lambda, \nu) \in \NN \, . $$
Soit $\nu = \sum_i (1-r+r_i) n_i \: \! \omega_i \in P'$, on note $\nu' := \sum_i n_i \: \! {}^L \omega_i$, où les ${}^L \omega_i$ sont les poids fondamentaux de ${}^L \Glie$. En remplaçant dans $\chi (L(\lambda))$ chaque $e^{\nu}$ par $e^{\nu'}$ lorsque $\nu \in P'$ et par $0$ sinon, il est montré dans \cite{hernandez} que l'on obtient le caractère ${}^L \chi ( {}^L \! L(\lambda))$ d'une représentation virtuelle de ${}^L \Glie$, qui contient la représentation irréductible $L(\lambda')$ de ${}^L \Glie$, i.e. $\chi(L\! (\lambda')) \leq {}^L \chi ( {}^L L(\lambda))$. Pour une description de la dualité de Langlands au niveau des caractères dans le cas général d'une algèbre de Kac-Moody, voir \cite{mcgerty}. \\

Dans le cas d'une algèbre $\Glie$ de type fini Littelmann \cite{littelmann} a prouvé que la représentation virtuelle de ${}^L \Glie$ est en fait une représentation ${}^L \! L(\lambda)$ réelle. Dans le cas général d'une algèbre de Kac-Moody, voir McGerty \cite{mcgerty}. \\
La stratégie consiste à $q$-déformer la situation. Les théories des représentations de $\Glie$ et du groupe quantique $U_q(\Glie)$ (avec $q$ générique) sont semblables, en particulier il existe une représentation $L^q(\lambda)$ de $U_q(\Glie)$, analogue de la représentation $L(\lambda)$ de $\Glie$. Lorsque $q$ est spécialisé à une certaine racine de l'unité $\varepsilon$, l'homomorphisme de Frobenius quantique défini par Lusztig \cite[35]{lusztig} peut être étendu afin de construire une action du groupe quantique modifiée $\dot{U}_{\varepsilon} ( {}^L \Glie)$ sur $L^{\varepsilon} (\lambda)$, qui fournit alors l'action de ${}^L \Glie$ sur ${}^L \! L(\lambda)$. \\

La dualité de Langlands décrite précédemment a des liens avec la correspondance géométrique de Langlands (voir \cite{frenkelcft} pour une introduction générale au programme de Langlands et \cite{hernandez} pour d'autres références). On suit ici l'introduction de \cite{hernandez}. \\
Soit $\Glie$ une algèbre de Lie semi-simple de dimension finie. Un résultat clé dans la correspondance géométrique de Langlands est un isomorphisme entre le centre $Z(\hat{\Glie})$ de l'algèbre enveloppante complétée de $\hat \Glie$ au niveau critique et la $\mathcal{W}$-algèbre classique $\mathcal{W}({}^L \Glie)$. Afin de mieux comprendre cet isomorphisme, on peut $q$-déformer la situation et considérer le centre $Z_q(\hat{\Glie})$ de l'algèbre affine quantique $U_q(\hat{\Glie})$ au niveau critique, ce qui est le point de départ dans \cite{frenkel-re}. Le centre $Z_q(\hat{\Glie})$ apparaît alors relié à l'anneau de Grothendieck $\text{Rep} \, U_q(\hat{\Glie})$ des représentations de dimension finie de $U_q(\hat{\Glie})$. On espère donc obtenir un nouvel éclairage sur l'isomorphisme $Z(\hat{\Glie}) \simeq \mathcal{W}({}^L \Glie)$ en étudiant les connections entre $Z_q(\hat{\Glie})$, $\text{Rep} \, U_q(\hat{\Glie})$ et la $\mathcal{W}$-algèbre classique $q$-déformée. \\
L'idée de Frenkel et Reshetikhin \cite{frenkel} était de déformer un nouvelle fois la situation, en introduisant une double déformation $\mathcal{W}_{q,t}$ de $Z(\hat{\Glie})$. La spécialisation $\mathcal{W}_{q,1}$ à $t=1$ est le centre $Z_q(\hat{\Glie})$ et ils ont suggéré que la spécialisation $\mathcal{W}_{\varepsilon,t}$ à $q=\varepsilon$, où $\varepsilon$ désigne une certaine racine de l'unité, contient le centre $Z_t({}^L \hat{\Glie})$ de l'algèbre affine quantique $U_t({}^L \hat{\Glie})$ au niveau critique. $Z_t({}^L \hat{\Glie})$ étant lié à l'anneau de Grothendieck $\text{Rep} \, U_t({}^L \hat{\Glie})$ des représentations de dimension finie de $U_t({}^L \hat{\Glie})$, la $\mathcal{W}$-algèbre $\mathcal{W}_{q,t}$ apparaît ainsi interpoler les anneaux de Grothendieck des représentations de dimension finie des algèbres affines quantiques associées à $\hat \Glie$ et ${}^L \hat{\Glie}$. En particulier cela suggère que ces représentations doivent être d'une certaines manières liées. On peut trouver dans \cite{frenkel} des exemples de telles relations. \\

C'est, motivés par les questions précédentes, que Frenkel et Hernandez \cite{hernandez} introduisent la dualité de Langlands pour les représentations de dimension finie de $\Glie$ et ${}^L \Glie$, dans la situation où $\Glie$ est de type finie. \\
Cette dualité, inspirée par la correspondance de Langlands géométrique, se trouve être exactement la dualité (sans rapport a priori avec le programme de Langlands) établie dix ans plus tôt par Littelmann \cite{littelmann}.  L'approche dans \cite{hernandez} est toutefois différente de celle de Littelmann. Dans l'esprit de \cite{frenkel}, ils ont introduit une double déformation $\widetilde{U}_{q,t} (\Glie)$ de $\widetilde{U}(\Glie)$, l'algèbre enveloppante de $\Glie$ sans les relations de Serre, et conjecturé que cette double déformation permettait d'expliquer la dualité de Langlands pour les représentations de $\Glie$ et ${}^L \! \Glie$. \\

Historiquement, les groupes quantiques ont été définis par Drinfeld \cite{drinfeld, drinfeld-icm} et Jimbo \cite{Jimbo}. Drinfeld les définit comme des déformations formelles $U_h(\Glie)$ des algèbres enveloppantes $U(\Glie)$, tandis que Jimbo définit leurs versions rationnelles $U_q(\Glie)$. L'algèbre $U_h(\mathfrak{sl}_2)$ avait été construite par Kulish et Reshetikhin \cite{reshetikhin}. \\
Le terme déformation formelle a dans ce cas la signification suivante : à une algèbre $A$ sur $\CC$, on associe une $\CC[[h]]$-algèbre que l'on note $A[[h]]$. En tant que $\CC[[h]]$-module, $A[[h]]$ est l'ensemble des séries formelles à coefficients dans $A$ :
$$ \sum_{n \in \NN} a_n \: \! h^n \, , \quad \text{avec } a_n \in A \, . $$
C'est alors le produit de $A[[h]]$ qui porte la déformation, dans le sens où l'on doit retrouver celui de $A$ quand $h = 0$. On pourra consulter à ce sujet \cite[III]{guichardet}, ainsi que \cite[XVI]{kassel} et \cite[6.1]{chari}. \\

Les groupes quantiques d'interpolation $\widetilde{U}_{q,t} (\Glie)$ introduits par Frenkel et Hernandez dépendent de deux paramètres $q$ et $t$. Cet article propose de prendre le point de vue des déformations formelles. Pour cela, il nous faudra considérer des doubles déformations d'algèbres, selon des paramètres formels $h$ et $h'$. À une algèbre $A$, on associe donc une algèbre $A[[h,h']]$, dont les éléments sont maintenant des séries formelles à deux indéterminées et à coefficients dans $A$
$$ \sum_{n,m \in \NN} a_{n,m} \, h^n (h')^m \, , \quad \text{avec } a_{n,m} \in A \, . $$
On construit dans cet article une classe, paramétrée par un entier $g \in \NN_{\geq 1}$, de déformations formelles $\Uhh$ de $U(\mathfrak{sl}_2)$ selon les paramètres $h$ et $h'$. Ces déformations sont les groupes quantiques d'interpolation de Langlands de rang $1$ (i.e. pour $\Glie = \mathfrak{sl}_2$). Il s'agit d'une première étape dans la construction de groupes quantiques d'interpolation de Langlands $\widetilde{U}_{h,h'}(\Glie,g)$ associés à une algèbre de Kac-Moody (symétrisable), qui apparaîtront dans un prochain article. \\

L'un des objectifs des groupes quantiques d'interpolation de Langlands est de donner une explication de la dualité de Langlands pour les représentations d'une algèbre de Kac-Moody $\Glie$ et de sa duale ${}^L \Glie$. \\
On attend que, pour certaines valeurs de $g$, $\widetilde{U}_{h,h'}(\Glie,g)$ interpolent les groupes quantiques $\widetilde{U}_h(\Glie)$ et $\widetilde{U}_{g' h'}(^{L} \Glie)$, où $g'$ est un entier qui dépend de $g$ et $\Glie$ : la limite $h' \to 0$ de $\widetilde{U}_{h, h'}(\Glie,g)$ est $\widetilde{U}_h(\Glie)$ et la spécialisation à $Q := \exp h = \varepsilon$, avec $\varepsilon$ une racine $(2g)$-ième primitive de l'unité, admet comme sous-quotient $\widetilde{U}_{g' h'}(^{L} \Glie)$.
$$ \xymatrix{
&& \widetilde{U}_{h,h'}(\Glie,g) \ar[lld]_-{\lim_{h' \to 0 \ }} \ar[rrd]^-{\lim_{Q \to \varepsilon}} && \\
\widetilde{U}_h(\Glie) && && \widetilde{U}_{g' h'}({}^L \Glie)
} $$
Dans cet article, cette conjecture est résolue au rang 1 : voir le théorème \ref{thm_uhhinterpolation}. \\

L'interpolation ci-dessus induit une correspondance $(\Glie,g) \longleftrightarrow ({}^L \Glie, g')$, qui rappelle la dualité $S$ pour les théories de jauge. Le paramètre $g$ y jouerait le rôle de la constante de couplage. On pourra consulter \cite{frenkel-bourbaki} à propos des théories de jauge et de la dualité de Langlands. \\

Pour $\lambda \in P'$, on conjecture (voir aussi \cite[conjecture 1]{hernandez}) l'existence d'une double déformation de $L(\lambda)$ où les actions de $\Glie$ et ${}^L \Glie$ seraient toutes deux contenues. Plus précisément on conjecture l'existence d'une représentation $L^{h,h'}(\lambda,g)$ de $\widetilde{U}_{h,h'}(\Glie)$, qui interpole la représentation $L^h(\lambda)$ et une représentation Langlands $g$-duale ${}^L \! L^{h'}(\lambda,g)$ : la limite $h' \to 0$ de $L^{h,h'}(\lambda,g)$ est la représentation $L^h(\lambda)$, et la spécialisation à $Q = \varepsilon$ contient la représentation ${}^L \! L^{g'h'}(\lambda,g)$.
\begin{equation} \label{eq_reprinter}
\xymatrix{
&& L^{h,h'}(\lambda,g) \ar[lld]_-{\lim_{h' \to 0 \ }} \ar[rrd]^-{\lim_{Q \to \varepsilon}} && \\
L^h(\lambda) && && \quad {}^L \! L^{h'}(\lambda,g) \ \supset \ L^{g'h'}(\lambda')
} \end{equation}
Cette conjecture est dans cet article résolue en rang 1 : voir le théorème \ref{thm_reprinter}. \\

Pour établir un parallèle avec les groupes quantiques usuels, on peut penser, dans le cas $\Glie$ de type fini et simple, à $\widetilde{U}_{q,t}(\Glie)$ comme une version spécialisée de $\widetilde{U}_{h,h'}(\Glie,r)$ (on rappelle que $r$ désigne la longueur maximale des racines de $\Glie$). Toutefois alors qu'on peut retrouver $U_q(\Glie)$ comme sous-algèbre de $U_h(\Glie)$, il n'est pas vrai que $\widetilde{U}_{q,t}(\Glie)$ apparaît comme sous-algèbre de $\widetilde{U}_{h,h'}(\Glie,r)$. Les groupes quantiques d'interpolation proposés par Frenkel et Hernandez sont en fait notablement différents : ils nécessitent notamment beaucoup plus de générateurs pour être définis, voir par exemple la remarque \ref{rem_hernandezgen}. \\
On trouvera dans \cite{kassel-biquant} un autre exemple de doubles déformations de structure algébrique provenant des algèbres de Lie. On peut également citer \cite{multi} où des groupes quantiques à plusieurs paramètres ont été définis. \\

On donnera tout au long de l'article des comparaisons entre les groupes quantiques d'interpolation définis ici et ceux définis dans \cite{hernandez}. Notons déjà que le principal avantage de cette nouvelle définition réside dans le fait qu'ils sont des déformations formelles de $\widetilde{U}(\Glie)$. Par ailleurs :
\begin{itemize}
\item[\textbullet] on obtient dans cet article une définition plus claire et plus maniable des groupes quantiques d'interpolation de Langlands;
\item[\textbullet] ce travail permettra, en toute généralité, de définir un groupe quantique d'interpolation associé à une algèbre de Kac-Moody $\Glie$ symétrisable et à un paramètre $g \in \NN_{\geq 1}$ (dans un prochain article);
\item[\textbullet] les outils et résultats de la théorie des déformations formelles sont, dans notre cadre, utilisables pour l'étude des groupes quantiques d'interpolation.
\end{itemize}

\subsection*{}
On décrit maintenant l'organisation de l'article. Notons que tous les résultats présentés à partir de la section \ref{section_def} sont nouveaux.

\begin{itemize}
\item[\textbullet] Dans la section \ref{section_deformations}, on expose une partie de la théorie des déformations formelles dans le cadre de plusieurs paramètres de déformation. Les théorèmes \ref{thm_hoch} et \ref{thm_hochrep} concluent cette section.

\item[\textbullet] Dans la section \ref{section_def}, on donne et on explique la définition des groupes quantiques d'interpolation $\Uhh$.

\item[\textbullet] Dans la section \ref{section_prop}, on établit différentes propriétés des groupes quantiques d'interpolation de $\Uhh$. On prouve notamment qu'ils sont des déformations formelles triviales de $\Uc$ selon $h$ et $h'$, de $\Uh$ selon $h'$ et de $\Urh{\: \! h'}$ selon $h$ : voir le théorème \ref{thm_uhhdefor}. \\
Ces différentes déformations peuvent être représentées par le diagramme commutatif
$$ \xymatrix{
&& \Uhh \ar[dll]_-{\lim_{h' \to 0} \ } \ar[drr]^-{\lim_{h \to 0}} \ar[dd]^-{\lim_{h,h' \to 0}} && \\
\Uh \ar[drr]_-{\lim_{h \to 0}} && && \Urh{\: \! h'} \ar[dll]^-{\ \lim_{h' \to 0}} \\
&& \Uc &&
} $$
Deux des points techniques clés de cet article sont la proposition \ref{prop_uconst}, ainsi que le lemme \ref{lem_fond}, démontrés respectivement au tout début et à la fin de la section \ref{section_prop}.

\item[\textbullet] Dans la section \ref{section_repr} on étudie la théorie des représentations de rang fini de $\Uhh$, qui s'avère être similaire à celle de $\Uc$ et $\Uh$. On prouve l'existence de doubles déformations pour tout $g \in \NN_{\geq 1}$ des représentations irréductibles de dimension finie et des modules de Verma de $\mathfrak{sl}_2$ : voir les propositions \ref{prop_vermadescr} et \ref{prop_indecdescr}. \\
Le théorème \ref{thm_rephh} établit plusieurs propriétés de la catégorie des représentations de rang fini de $\Uhh$. Notamment, on prouve qu'elle possède la propriété de Krull-Schmidt et on classifie ces objets indécomposables.

\item[\textbullet] Dans la section \ref{section_langlands}, on établit l'existence d'une double déformation $L^{h,h'}(n,g)$ possédant la propriété d'interpolation \eqref{eq_reprinter}, résolvant ainsi la conjecture \cite[conjecture 1]{hernandez} en rang 1 pour tout $g \in \NN_{\geq 1}$ : voir le théorème \ref{thm_reprinter}. \\
Le théorème \ref{thm_reprinter2} donne un résultat analogue pour les modules de Verma. \\
Enfin on prouve que $\Uhh$ interpole les groupes quantiques $\Uh$ et $\Urh{g'h'}$ : voir le théorème \ref{thm_uhhinterpolation}.
$$ \xymatrix{
&& \Uhh \ar[lld]_-{\lim_{h' \to 0 \ }} \ar[rrd]^-{\lim_{Q \to \varepsilon}} && \\
\Uh && && \Urh{g'h'}
} $$
\end{itemize}

\section*{Conventions}
Toutes les algèbres et anneaux considérés seront supposés associatifs et unitaires, avec $1 \neq 0$. Un morphisme d'algèbre (ou d'anneau) envoie l'unité sur l'unité. \\
$\Bbbk$, sans précision supplémentaire, désignera un anneau commutatif. \\
Si $X$ est un ensemble quelconque, fini ou infini, $\langle X \rangle$ désignera le monoïde libre engendré par $X$, i.e. l'ensemble des mots (y compris le mot vide), dont les lettres sont des éléments de $X$, muni de la concaténation comme opération. $\Bbbk \langle X \rangle$ désigne la $\Bbbk$-algèbre du monoïde $\langle X \rangle$. C'est l'algèbre associative unitaire libre engendrée par $X$. \\
Une somme indexée par un ensemble vide est nulle, un produit indexé par un ensemble vide vaut $1$.

\section{Déformations Formelles} \label{section_deformations}
La théorie des déformations formelles d'algèbres a été introduite par Gerstenhaber dans \cite{gerstenhaber}. La théorie est construite principalement pour un seul paramètre de déformation, bien que dans \cite{gerstenhaber} il est indiqué de quelle manière certains résultats peuvent être généralisés. On se propose de donner dans cette section une présentation d'une partie de la théorie dans le cadre de plusieurs paramètres de déformation. L'exposition qui suit est une généralisation naturelle de celles que l'on peut trouver, pour un seul paramètre, dans par exemple \cite[III]{guichardet}, ainsi que \cite[XVI]{kassel} et \cite[6.1]{chari}. On prendra soin de détailler les arguments spécifiques à la généralisation, les autres seront rappelés.

\subsection{Premières définitions}
Soit un entier $s \geq 1$. On note $\Bbbk[[h_1,h_2, \dots, h_s]]$, ou encore $\Bbbk[[\bar{h}]]$ la $\Bbbk$-algèbre commutative des séries formelles à $s$ indéterminées $h_1, \dots, h_s$. On note $\Bbbk((h_1, h_2, \dots, h_s))$, ou encore $\Bbbk((\bar{h}))$, le corps de fonctions de $\Bbbk[[\bar{h}]]$. $\Bbbk((\bar{h}))$ est l'algèbre des séries de Laurent à $s$ indéterminées $h_1,h_2, \dots, h_s$, c'est aussi le localisé de $\Bbbk[[\bar{h}]]$ par rapport aux éléments $h_1,h_2, \dots, h_s$. \\ \\
$\Bbbk[[\bar{h}]]$ est une algèbre graduée en considérant le degré total d'un monôme
$$ \text{deg} \big( h_1^{n_1} h_2^{n_2} \cdots h_s^{n_s} \big) \ := \ n_1 + n_2 + \cdots + n_s \, . $$
On note par ailleurs
$$ (\bar{h})  := \big( h_1, h_2, \dots, h_s \big) $$
l'idéal de $\Bbbk[[\bar h]]$ engendré par les éléments $h_1,h_2, \dots, h_s$. \\
La filtration associée à la graduation est également celle induite par l'idéal $(\bar h)$ et fait de $\Bbbk[[\bar{h}]]$ une algèbre complète et séparée. On l'appelle la filtration $\bar h$-adique (la topologie associée est appelée la topologie $\bar h$-adique). \\
La filtration sur un $\Bbbk[[\bar h]]$-module $M$, induite par celle de $\Bbbk[[\bar{h}]]$, est aussi appelée la filtration $\bar h$-adique (la topologie associée sur $M$ est appelée la topologie $\bar h$-adique). \\

Pour $V$ un $\Bbbk$-module, on définit un nouveau $\Bbbk$-module, que l'on note $V[[h_1, h_2, \dots, h_s]]$ ou encore $V[[\bar{h}]]$, comme l'espace des séries formelles à coefficients dans $V$
$$ \sum_{n_1, n_2, \dots, n_s \, \geq \, 0} v_{n_1,n_2, \dots, n_s} \, h_1^{n_1} h_2^{n_2} \cdots h_s^{n_s} \, , \quad \text{avec } \ v_{n_1,n_2, \dots, n_s} \in V . $$
Pour alléger l'écriture, introduisons les notations suivantes : $\bar n = (n_1, n_2, \dots, n_s)$ désignera un $s$-uplet de $\NN^s$, et $h^{\bar{n}}$ le monôme $h_1^{n_1} h_2^{n_2} \cdots h_s^{n_s}$. \\
Un élément de $V[[\bar{h}]]$ peut alors s'écrire sous la forme suivante :
$$ \sum_{\bar{n} \in \NN^s} v_{\bar n} \: \! h^{\bar{n}} \, , \quad \text{avec } \ v_{\bar{n}} \in V . $$

L'espace vectoriel $V[[\bar{h}]]$ est naturellement muni d'une structure de $\Bbbk[[\bar{h}]]$-module :
$$ \Big( \sum_{\bar{p} \in \NN^s} x_{\bar p} \, h^{\bar{p}} \Big) . \Big( \sum_{\bar{q} \in \NN^s} v_{\bar q} \, h^{\bar{q}} \Big) \ := \ \sum_{\bar{n} \in \NN^s} \Bigg( \sum_{\bar{p} + \bar{q} = \bar{n}} x_{\bar{p} \, } .v_{\bar q} \Bigg) h^{\bar{n}} \, , $$
avec $x_{\bar{p}} \in \Bbbk$ et $v_{\bar{q}} \in V$. \\
On définit une structure de $\Bbbk[[\bar h]]$-module gradué sur $V[[\bar h]]$ en posant
$$ \text{deg} \big( h_1^{n_1} h_2^{n_2} \cdots h_s^{n_s} \big) \ := \ n_1 + n_2 + \cdots + n_s \, . $$
La filtration associée à cette graduation est la filtration $\bar h$-adique et fait de $V[[\bar{h}]]$ un $\Bbbk[[\bar{h}]]$-module complet et séparé. \\

Le localisé du $\Bbbk[[\bar{h}]]$-module $V[[\bar{h}]]$ par rapport aux éléments $h_1, h_2, \dots, h_s$ est un espace vectoriel sur $\Bbbk((\bar{h}))$, que l'on notera $V((h_1, h_2, \dots, h_s))$, ou encore $V((\bar{h}))$. Il s'agit de l'espace des séries de Laurent à coefficients dans $V$
$$ \sum_{\bar{n} \, \succeq - \overline{N}} v_{\bar n} \: \! h^{\bar{n}} \, , \quad \text{avec } \ v_{\bar{n}} \in V , \text{ et } \overline{N} \in \NN^s . $$
L'ordre partiel $\preceq$ sur $\ZZ^s$ que l'on utilise est défini par $\bar{n} \preceq \bar{m}$ si $\bar{m} - \bar{n} \in \NN^s$. \\

Si $A$ est une $\Bbbk$-algèbre (pas nécessairement commutative), on peut munir $A[[\bar{h}]]$ d'une structure de $\Bbbk[[\bar{h}]]$-algèbre :
$$ \Big( \sum_{\bar{p} \in \NN^s} a_{\bar p} \, h^{\bar{p}} \Big) \cdot \Big( \sum_{\bar{q} \in \NN^s} b_{\bar q} \, h^{\bar{q}} \Big) \ := \ \sum_{\bar{n} \in \NN^s} \Bigg( \sum_{\bar{p} + \bar{q} = \bar{n}} a_{\bar{p} \, } b_{\bar q} \Bigg) h^{\bar{n}} \, , $$
où $a_{\bar{p}}$ et $b_{\bar{q}}$ sont des éléments de $A$. \\

On rappelle que $\Bbbk \langle X \rangle$ désigne la $\Bbbk$-algèbre du monoïde libre $\langle X \rangle$.

\begin{definition}
Soit $X$ un ensemble et $\mathcal{R}$ une partie de $\Bbbk \langle X \rangle[[\bar{h}]]$. La $\Bbbk[[\bar{h}]]$-algèbre topologiquement engendrée par $X$ et définie par les relations $\mathcal{R}$, que l'on note
$$ \Bbbk \langle X \rangle[[\bar{h}]] \, / \, \overline{(\mathcal{R})}^{\: \! \bar h} \, , $$
est le quotient de $\Bbbk \langle X \rangle[[\bar{h}]]$ par l'adhérence (topologique) de l'idéal bilatère engendré par $\mathcal{R}$ (ou encore le plus petit idéal bilatère fermé contenant $\mathcal{R}$).
\end{definition}

La filtration induite sur le quotient $\Bbbk \langle X \rangle[[\bar{h}]] \, / \, \overline{(\mathcal{R})}^{\: \! \bar h}$ par celle de $\Bbbk \langle X \rangle[[\bar{h}]]$ est la filtration $\bar h$-adique, et fait de $\Bbbk \langle X \rangle[[\bar{h}]] \, / \, \overline{(\mathcal{R})}^{\: \! \bar h}$ une algèbre complète et séparée. La topologie $\bar h$-adique de $\Bbbk \langle X \rangle[[\bar{h}]] \, / \, \overline{(\mathcal{R})}^{\: \! \bar h}$ est la topologie quotient. \\

Remarquons les propriétés suivantes. Leur démonstration est immédiate.

\begin{proposition} \label{prop_props}
\begin{itemize}
\item[1)] Si $V$ et $W$ sont deux $\Bbbk$-modules, on a un isomorphisme canonique de $\Bbbk[[\bar{h}]]$-module entre $\text{Hom}_{\Bbbk[[\bar{h}]]} \big( V[[\bar{h}]],W[[\bar{h}]] \big)$ et $\text{Hom}_{\Bbbk}(V,W)[[\bar{h}]]$.

\item[2)] Dans le cas où $V = W$, l'isomorphisme précédent entre $\text{End}_{\Bbbk[[\bar{h}]]} \big( V[[\bar{h}]] \big)$ et $\left( \text{End}_{\Bbbk} \: \! V \right) [[\bar{h}]]$ est un isomorphisme de $\Bbbk[[\bar{h}]]$-algèbre.

\item[3)] Si $f : X \to B$ est une application de $X$ dans une $\Bbbk[[\bar{h}]]$-algèbre $B$ complète et séparée (pour la filtration $\bar{h}$-adique), alors il existe un unique morphisme de $\Bbbk[[\bar{h}]]$-algèbre $\tilde{f} : \Bbbk \langle X \rangle[[\bar{h}]] \to B$ tel que $\tilde{f}(x) = f(x)$ pour tout $x \in X$.

\item[4)] Si $\mathcal{R}$ est une partie de $\Bbbk \langle X \rangle[[\bar{h}]]$, l'ensemble des applications $f : X \to B$ telles que $\tilde f$ s'annule sur $\mathcal R$ est en bijection avec l'ensemble des morphismes de $\Bbbk[[\bar{h}]]$-algèbre de $\Bbbk \langle X \rangle[[\bar{h}]] \, / \, \overline{(\mathcal{R})}^{\: \! \bar h}$ vers $B$.
\end{itemize}
\end{proposition}

\subsection{Déformations d'algèbres}
Soit $A$ est une $\Bbbk$-algèbre. Le produit de $A[[\bar h]]$ étend celui de $A$ au sens suivant : l'isomorphisme $\Bbbk$-linéaire canonique entre $A[[\bar{h}]] / (\bar{h} = 0)$ (on quotiente par l'idéal bilatère de $A[[\bar h]]$ engendré par les éléments $h_1, h_2, \dots, h_s$) et $A$ est un isomorphisme de $\Bbbk$-algèbre. \\
Il peut exister d'autres produits sur $A[[\bar{h}]]$ satisfaisant cette propriété, ce qui motive la définition suivante.

\begin{definition}
On appelle déformation formelle (à $s$ paramètres) d'une $\Bbbk$-algèbre $A$ une structure de $\Bbbk[[\bar{h}]]$-algèbre sur le $\Bbbk[[\bar{h}]]$-module $A[[\bar{h}]]$ telle que $1 \in A$ est l'élément neutre de $A[[\bar{h}]]$, et telle que l'isomorphisme $\Bbbk$-linéaire canonique entre $A[[\bar{h}]] / (\bar{h} = 0)$ et $A$ est un isomorphisme de $\Bbbk$-algèbre.
\end{definition}

Soit $V$ un $\Bbbk$-module. La donnée d'une application $\Bbbk[[\bar{h}]]$-bilinéaire $u$ de $V[[\bar{h}]] \times V[[\bar{h}]]$ vers $V[[\bar{h}]]$ est équivalente à la donnée d'une famille $(u_{\bar{n}})_{\bar{n} \in \NN^s}$ d'applications $\Bbbk$-bilinéaires de $V \times V$ dans $V$. Plus explicitement, on a :
$$ u \bigg( \sum_{\bar{p} \in \NN^s} v_{\bar{p}} \, h^{\bar{p}} \, , \sum_{\bar{q} \in \NN^s} w_{\bar{q}} \, h^{\bar{q}} \bigg) \ = \ \sum_{\bar{n} \in \NN^s} \Bigg( \sum_{\bar{p} + \bar{q} + \overline{r} = \bar{n}} u_{\overline{r}} \big( v_{\bar{p}}, w_{\bar{q}} \big) \Bigg) h^{\bar{n}} \, . $$

Notons $\mu$ le produit d'une $\Bbbk$-algèbre $A$ et $\mu_{\bar h}$ une application $\Bbbk[[\bar{h}]]$-bilinéaire sur $A[[\bar{h}]]$. En notant $\mu_{\bar n}$ les applications $\Bbbk$-bilinéaires associées à $\mu_{\bar h}$, l'associativité de $\mu_{\bar h}$ est équivalente à la condition suivante :
\begin{equation} \label{eq_assoc}
\sum_{\bar{p} + \bar{q} = \bar{n}} \mu_{\bar{p}} \Big( \mu_{\bar{q}} (a,b), c \Big) - \, \mu_{\bar{p}} \Big( a, \mu_{\bar{q}}(b,c) \Big) \ = \ 0 \, , \quad \forall \, \bar{n} \in \NN^s, \ \forall \, a,b,c \in A \, . 
\end{equation}
Supposons que $\mu_{\bar h}$ munisse $A[[\bar{h}]]$ d'une structure de $\Bbbk[[\bar{h}]]$-algèbre (on suppose donc l'associativité de $\mu_{\bar h}$ et que $1 \in A$ est un élément neutre de $\mu_{\bar h}$), alors cette structure est une déformation formelle de $A$ si et seulement si $\mu_0 = \mu$. \\

On appelle déformation formelle constante la structure de $\Bbbk[[\bar{h}]]$-algèbre canonique sur $A[[\bar{h}]]$, caractérisée par $\mu_{\bar{n}} = 0$ pour $\bar{n} \succ 0$.

\begin{definition}
Deux déformations formelles d'une $\Bbbk$-algèbre $A$ sont dites équivalentes si il existe un isomorphisme de $\Bbbk[[\bar{h}]]$-algèbre $\phi : A[[\bar{h}]] \to A[[\bar{h}]]$ entre les deux structures, qui induit l'identité sur $A[[\bar{h}]] / (\bar{h} = 0)$.
\end{definition}

On dira qu'une déformation formelle est triviale si elle est équivalente à la déformation formelle constante.

\subsection{Déformations de représentations}
\begin{definition}
Soit $A$ une $\Bbbk$-algèbre. On appelle représentation d'une déformation formelle de $A$, que l'on note $(V,\pi_{\bar h})$, la donnée d'un $\Bbbk$-module $V$ et d'un morphisme de $\Bbbk[[\bar{h}]]$-algèbre $\pi_{\bar h} : A[[\bar h]] \to \text{End}_{\Bbbk[[\bar{h}]]} (V[[\bar h]])$, où $A[[\bar{h}]]$ est considéré avec sa déformation formelle. \\
 $(V,\pi_{\bar h})$ est dit de rang fini (resp. de rang $r \in \NN$) si $V$ est un module libre de rang fini (resp. de rang $r$) sur $\Bbbk$.
\end{definition}

On appellera sous-représentation d'une représentation $V[[\bar h]]$ une représentation de la forme $W[[\bar h]]$, avec $W$ un sous-$\Bbbk$-module de $V$. On définit le quotient de $V[[\bar h]]$ par $W[[\bar h]]$ comme étant la représentation $(V/W)[[\bar h]]$. \\

Notons $\mu_{\bar h}$ le produit d'une déformation formelle de $A$. En utilisant le point 1) de la proposition \ref{prop_props}, on voit que la donnée d'une représentation $(V,\pi_{\bar h})$ est équivalente à la donnée d'une famille $\big( \pi_{\bar{n}} \big)_{\bar{n} \in \NN^s}$ d'applications $\Bbbk$-linéaires de $A$ dans $\text{End}_{\Bbbk}(V)$ qui vérifient :
\begin{equation} \label{eq_reprcg}
\sum_{\bar{p} + \bar{q} = \bar{n}} \Big[ \pi_{\bar{p}} \big( \mu_{\bar{q}}(a,b) \big)  \, - \ \pi_{\bar{p}}(a) \circ \pi_{\bar{q}}(b) \Big] \ = \ 0 \, , \quad \ \forall \, a, b \in A \, , \ \forall \, \bar{n} \in \NN^s \, .
\end{equation}
On voit que la représentation $(V,\pi_{\bar h})$ induit une représentation $(V,\pi_0)$ de $A$. On dira que $(V,\pi_{\bar h})$ est une déformation formelle d'une représentation $(V,\pi)$ de $A$ si $\pi_0 = \pi$.  \\

On fera la distinction entre représentation d'une déformation formelle et module d'une déformation formelle : un module d'une déformation formelle $A[[\bar h]]$ de $A$ désignera un $A[[\bar{h}]]$-module. En d'autres mots, une représentation est la donnée d'une structure de $A[[\bar{h}]]$-module sur un espace de la forme $V[[\bar{h}]]$. \\

En utilisant le point 1) encore de la proposition \ref{prop_props}, on voit que deux représentations $(V,\pi_{\bar h})$ et $(W,\rho_{\bar h})$ sont isomorphes en tant que $A[[\bar h]]$-modules, si et seulement si il existe une famille $\big( u_{\bar{n}} \big)_{\bar{n} \in \NN^s}$ d'applications $\Bbbk$-linéaires de $V$ vers $W$ telle que $u_0$ est bijective et telle que
\begin{equation}
\sum_{\bar{p} + \bar{q} = \bar{n}} \Big[ u_{\bar{p}} \circ \pi_{\bar{q}}(a) \, - \ \rho_{\bar{p}}(a) \circ u_{\bar{q}} \Big] \ = \ 0 \, , \quad \ \forall \, a \in A \, , \ \forall \, \bar{n} \in \NN^s \, .
\end{equation}
On voit que si les deux représentations sont isomorphes, $u_0$ est un isomorphisme entre les représentations de $A$ induites $(V,\pi_0)$ et $(W,\rho_0)$. \\
On dira que deux représentations $(V,\pi_{\bar h})$ et $(W,\rho_{\bar h})$ sont équivalentes si l'on ajoute $u_0 = \text{id}$ aux conditions précédentes. \\

Considérons la déformation formelle constante de $A$. Dans ce cas, d'après \eqref{eq_reprcg}, la donnée d'une représentation $(V,\pi_{\bar h})$ est équivalente à la donnée d'une famille $\big( \pi_{\bar{n}} \big)_{\bar{n} \in \NN^s}$ d'applications $\Bbbk$-linéaires de $A$ dans $\text{End}_{\Bbbk}(V)$ qui vérifient
\begin{equation} \label{eq_repr}
\pi_{\bar{n}} (ab) \ = \ \sum_{\bar{p} + \bar{q} = \bar{n}} \, \pi_{\bar{p}}(a) \circ \pi_{\bar{q}}(b) \, , \quad \ \forall \, a, b \in A \, , \ \forall \, \bar{n} \in \NN^s \, .
\end{equation}
La déformation formelle constante d'une représentation $(V,\pi)$ de $A$ est définie par $\pi_0 = \pi$ et $\pi_{\bar{n}} = 0$ pour $\bar{n} \succ 0$. \\
Une déformation formelle sera dite triviale si elle est équivalente à une déformation formelle constante. \\

On peut définir de manière similaire la notion de bimodule $M[[\bar h]]$ d'une déformation formelle $A[[\bar h]]$. Toutes les définitions précédentes s'adaptent naturellement. \\

Notons que si $V$ est une représentation d'une $\Bbbk$-algèbre $A$. Alors $\text{End}_{\Bbbk} V$ est naturellement muni d'une structure de $A$-bimodule :
$$ (a.f) (x) \ := \ a. \big( f(x) \big) \, , \quad \ (f.a) (x) \ := \ f(a.x) \, , \quad \text{ avec } \ f \in \text{End}_{\Bbbk} V, \ a \in A, \ x \in V \, . $$

La démonstration du lemme suivant est immédiate.

\begin{lemma} \label{lem_defor_end}
Soient $A$ une $\Bbbk$-algèbre et $V$ une représentation de $A$. Si $V[[\bar h]$ est une déformation formelle triviale de $V$, alors $\text{End}_{\Bbbk[[\bar h]]} (V[[\bar h]])$ est une déformation formelle triviale du $A$-bimodule $\text{End}_{\Bbbk} V$.
\end{lemma}

Soit $A$ une $\Bbbk$-algèbre et considérons une déformation formelle $A[[\bar h]]$ de $A$. \\
On note $\mathcal{C}(A)$ la catégorie abélienne des représentations de $A$ et $\mathcal{C}^{\bar{h}}(A)$ la catégorie des représentations de la déformation formelle $A[[\bar{h}]]$. \\

On vérifie que $\mathcal{C}^{\bar{h}}(A)$ est une catégorie additive $\Bbbk[[\bar h]]$-linéaire. \\
La somme directe de deux représentations $V[[\bar h]], W[[\bar h]] \in \mathcal{C}^{\bar{h}}(A)$ est le $\CC[[\bar h]]$-module $(V \oplus W)[[\bar h]]$, qu'on munit naturellement d'une structure de représentation de $A[[\bar h]]$. \\
Une représentation indécomposable est un objet indécomposable de la catégorie $\mathcal{C}^{\bar{h}}(A)$, i.e. une représentation qui n'est pas isomorphe à la somme directe de deux représentations. \\

L'existence d'un conoyau d'un morphisme $f : V[[\bar{h}]] \to W[[\bar{h}]]$ dans $\mathcal{C}^{\bar{h}}(A)$ n'est en général pas vérifiée : par exemple si $V = W = \Bbbk$ et $f$ est la multiplication par $h$. Par conséquent $\mathcal{C}^{\bar{h}}(A)$ n'est pas une catégorie abélienne. \\

La catégorie dont les objets sont les $\Bbbk[[\bar h]]$-modules de la forme $V[[\bar h]]$, et les flèches les morphismes de $\Bbbk[[\bar h]]$-module, est d'après ce qui précède une catégorie additive $\Bbbk[[\bar h]]$-linéaire (c'est le cas particulier où $A = \Bbbk$ et $\Bbbk[[\bar h]]$ est muni de la déformation formelle constante). On montre que cette catégorie est une catégorie tensorielle, où le produit tensoriel est défini pour $V, W$ deux $\Bbbk$-modules, par
$$ V[[\bar h]] \ \tilde{\otimes} \ W[[\bar h]] \ := \ (V \otimes_{\Bbbk} W) [[\bar h]] \, . $$
$V[[\bar h]] \ \tilde{\otimes} \ W[[\bar h]]$ est un complété $\bar h$-adique de $V[[\bar h]] \otimes_{\Bbbk[[\bar h]]} W[[\bar h]]$. \\

Notons $\mathcal{A}_{\Bbbk}$ la catégorie abélienne des $\Bbbk$-algèbres et $\mathcal{A}_{\Bbbk}^{\bar{h}}$ la catégorie $\Bbbk[[\bar h]]$-linéaire additive, dont les objets sont les déformations formelles de $\Bbbk$-algèbres, et les flèches les morphismes de $\Bbbk[[\bar h]]$-algèbre.  \\
On définit deux foncteurs $\Bbbk[[\bar h]]$-linéaires
$$ \mathcal{Q}^{\bar{h}} \, : \ \mathcal{A}_{\Bbbk} \ \longrightarrow \ \mathcal{A}_{\Bbbk}^{\bar{h}} \quad \text{ et } \quad \lim_{\bar{h} \to 0} \, : \ \mathcal{A}_{\Bbbk}^{\bar{h}} \longrightarrow \ \mathcal{A}_{\Bbbk} \, , $$
qui, respectivement, à une $\Bbbk$-algèbre $A$ associe sa déformation formelle constante, et à une déformation formelle de $A$ associe $A$ (les définitions de $\mathcal{Q}^{\bar{h}}$ et $\lim_{\bar{h} \to 0}$ pour les morphismes sont les définitions naturelles). Remarquons que
$$ \lim_{\bar{h} \to 0} \, \circ \ \mathcal{Q}^{\bar{h}} \ = \ \text{id}_{\! \mathcal{A}_{\Bbbk}} \, . $$

On définit de façon similaire un foncteur $\Bbbk[[\bar h]]$-linéaire de la catégorie des représentations d'une déformation formelle de $A$ vers la catégorie des représentations de $A$ :
$$ \lim_{\bar{h} \to 0} : \ \mathcal{C}^{\bar{h}}(A) \, \longrightarrow \ \mathcal{C}(A) \, . $$

On a également un foncteur $\Bbbk[[\bar h]]$

Le foncteur $\lim_{\bar h \to 0}$, que ce soit dans les cadre des algèbres ou celui des représentations, sera appelé limite classique.

\subsection{Cohomologie de Hochschild}
\begin{definition}
Soient $A$ une $\Bbbk$-algèbre et $M$ un $A$-bimodule. \\
Pour tout $n \in \NN$, on note $C_{\Bbbk}^n(A,M)$, et on appelle l'espace des $n$-cochaînes, le $\Bbbk$-module des applications $\Bbbk$-multilinéaires de $A^n$ dans $M$ (on a par convention $C_{\Bbbk}^0(A,M) = M$). \\
On définit une application $\Bbbk$-linéaire $d^{\: \! n} : C_{\Bbbk}^n(A,M) \to C_{\Bbbk}^{n+1}(A,M)$ par
\begin{eqnarray*}
(d^n f) (a_1, \dots, a_{n+1}) & := & \ \ \quad a_1 . f(a_2, \dots, a_{n+1}) \\
&& + \ \sum_{i=1}^n (-1)^i \, f(a_1, \dots, a_i \: \! a_{j+1}, \dots, a_{n+1}) \\
&& + \ (-1)^{n+1} f(a_1, \dots, a_n) . a_{n+1} \,
\end{eqnarray*}
(pour $n=0$ et $x \in M$, on a $(d^{\: \! 0} x) \: \! a = a.x - x.a$). \\
On appelle, et on note
$$ \Big( C_{\Bbbk}^{\bullet} (A,M),d \Big) := \Bigg( \bigoplus_{n \in \NN} \, C_{\Bbbk}^n(A,M) \, , \, \sum_{n \in \NN} \, d^{\: \! n} \Bigg) \, , $$
le complexe de Hochschild de $A$ à coefficients dans $M$ (on vérifie en effet que  $d \circ d = 0$). \\
La cohomologie de $C_{\Bbbk}^{\bullet}(A,M)$ est notée $H_{\Bbbk}^{\bullet}(A,M)$ et appelée la cohomologie de Hochschild de $A$ à coefficients dans $M$.
\end{definition}

\begin{rem} \label{rem_hoch}
Fixons la $\Bbbk$-algèbre $A$. La cohomologie de Hochschild $H^{\bullet}_{\Bbbk}(A, \cdot)$ définit un foncteur de la catégorie des $A$-bimodules vers la catégorie des $\Bbbk$-modules $\NN$-gradués.
\end{rem}

Le lemme suivant exprime le fait que la cohomologie de Hochschild commute en quelque sorte avec le processus de déformation formelle constante. Sa démonstration est immédiate.

\begin{lemma} \label{lem_hoch}
Soient $A$ une $\Bbbk$-algèbre et $M$ un $A$-bimodule. En considérant $A[[\bar h]]$ et $M[[\bar h]]$ comme déformations formelles constantes, on a pour tout $n \in \NN$ un isomorphisme canonique $\Bbbk[[\bar h]]$-linéaire :
$$ \Big( H_{\Bbbk}^n(A,M) \Big) [[\bar h]] \ \simeq \ H_{\Bbbk[[h]]}^n \big( A[[\bar h]],M[[\bar h]] \big) \, . $$
\end{lemma}

\begin{corollary} \label{cor_hoch}
Soient $A$ une $\Bbbk$-algèbre et $A[[\bar h]]$ une déformation formelle triviale de $A$. Il existe pour tout $n \in \NN$ un isomorphisme $\Bbbk[[\bar h]]$-linéaire
$$ \Big( H_{\Bbbk}^n(A,A) \Big) [[\bar h]] \ \simeq \ H_{\Bbbk[[h]]}^n \big( A[[\bar h]],A[[\bar h]] \big) \, . $$
\end{corollary}

\begin{proof}
$A[[\bar h]]$ étant isomorphe en tant que $\Bbbk[[\bar h]]$-algèbre à la déformation formelle constante de $A$, leurs cohomologies de Hochschild sont isomorphes. On conclut grâce au lemme \ref{lem_hoch}.
\end{proof}

\begin{corollary} \label{cor_hoch2}
Soient $A$ une $\Bbbk$-algèbre et $M[[\bar h]]$ une déformation formelle triviale d'un $A$-bimodule $M$. Il existe pour tout $n \in \NN$ un isomorphisme $\Bbbk[[\bar h]]$-linéaire
$$ \Big( H_{\Bbbk}^n(A,M) \Big) [[\bar h]] \ \simeq \ H_{\Bbbk[[h]]}^n \big( A[[\bar h]],M[[\bar h]] \big) \, , $$
où $A[[\bar h]]$ désigne la déformation formelle constante de $A$.
\end{corollary}

\begin{proof}
Conséquence du lemme \ref{lem_hoch} et de la remarque \ref{rem_hoch}.
\end{proof}

L'introduction du langage cohomologique s'avère intéressante comme en témoigne le théorème suivant, généralisation de \cite[2.5.3]{guichardet}. On vérifie que la démonstration dans \cite{guichardet} fonctionne encore quand on travaille sur un anneau commutatif $\Bbbk$ plutôt que sur $\CC$, puis on donne les arguments qui permettent de passer du cas d'un paramètre de déformation au cas de plusieurs paramètres.

\begin{theorem} \label{thm_hoch}
Soit $A$ une $\Bbbk$-algèbre et $s \geq 1$ un entier. Si $H_{\Bbbk}^2(A,A) = 0$ toute déformation formelle à $s$ paramètres de $A$ est triviale.
\end{theorem}

\begin{proof}
Nous allons procéder par récurrence sur $s$. Le cas $s = 1$ est l'énoncé de	\cite[2.5.3]{guichardet}, la seule différence étant que l'on travaille ici sur un anneau commutatif $\Bbbk$ (ce qui sera nécessaire pour la récurrence) et non sur $\CC$. On commence par montrer, sans supposer $H_{\Bbbk}^2(A,A) = 0$, que si $\mu_1 = \mu_2 = \cdots = \mu_{n-1} = 0$ (avec $n \in \NN_{\geq 1}$), et si $\mu_n$, qui est un $2$-cocycle en vertu de \eqref{eq_assoc}, est un cobord, alors il existe $\mu_h'$ équivalent à $\mu_h$ et vérifiant $\mu'_1 = \mu'_2 = \cdots = \mu'_n = 0$. \\
Il existe en effet par hypothèse $f \in \text{End}_{\Bbbk} A$ tel que
$$ \mu_n (a,b) \ = \ a \: \! f(b) \ - \ f(ab) \ + \ f(a) \: \! b \, . $$
On vérifie que l'automorphisme $\Bbbk[[h]]$-linéaire $\text{id}_A - h^n f$ conjugue $\mu_h$ avec une déformation ${\mu}'_h$ telle que
$$ \mu_h' (a,b) - ab \ \equiv \ 0 \quad \text{mod } h^{n+1} \, . $$
En supposant à présent $H_{\Bbbk}^2(A,A) = 0$ et $\mu_h$ quelconque, d'après ce qui précède on voit qu'il existe pour tout $p \in \NN_{\geq 1}$ un automorphisme $f_p$, tels que que le produit
\begin{equation*}
v_p \ := \ (\text{id}_A - h^p f_p) \circ \cdots \circ (\text{id}_A - h^2 f_2) \circ (\text{id}_A - h f_1)
\end{equation*}
conjugue $\mu_h$ avec une déformation $\mu_h^{(p)}$ telle que $\mu^{(p)}_1 = \mu^{(p)}_2 = \cdots = \mu^{(p)}_p = 0$. On conclut en remarquant que $v_p$ converge lorsque $p$ tend vers l'infini. \\

Supposons le théorème démontré pour $s \geq 1$ et considérons une déformation formelle à $s+1$ paramètres $A[[h_1, h_2 \dots, h_{s+1}]]$ de $A$. On note $\mu_{\bar h}$ le produit de cette déformation. \\
En considérant l'isomorphisme d'anneau canonique
$$ \Bbbk[[h_1, h_2 \dots, h_{s+1}]] \ \simeq \ \Big( \Bbbk[[h_1, h_2 \dots, h_s]] \Big) [[h_{s+1}]] $$
et l'isomorphisme linéaire canonique
$$ A[[h_1, h_2 \dots, h_{s+1}]] \ \simeq \ \Big( A[[h_1, h_2 \dots, h_s]] \Big) [[h_{s+1}]] $$
on voit que le produit $\mu_{\bar h}$ de $A[[h_1, h_2 \dots, h_{s+1}]]$ induit
\begin{itemize}
\item[\textbullet] sur $A[[h_1, h_2 \dots, h_s]]$ une structure de $\Bbbk[[h_1,h_2 \dots, h_s]]$-algèbre qui est une déformation formelle à $s$ paramètres de $A$;
\item[\textbullet] sur $\Big( A[[h_1, h_2 \dots, h_s]] \Big) [[h_{s+1}]]$ une structure de déformation à un paramètre de la $\Bbbk[[h_1, h_2, \dots, h_s]]$-algèbre $A[[h_1, h_2 \dots, h_s]]$ définie ci-dessus.
\end{itemize}
L'hypothèse de récurrence et le corollaire \ref{cor_hoch} permettent alors de conclure.
\end{proof}

Le théorème suivant généralise celui présenté dans \cite[III, 2.5.6]{guichardet}. De la même façon que pour le théorème \ref{thm_hoch}, la démonstration est une généralisation naturelle dans le cas $s=1$ et nécessite un argument supplémentaire pour passer au cas de plusieurs paramètres de déformation.

\begin{theorem} \label{thm_hochrep}
Soient $A$ une $\Bbbk$-algèbre et $(V, \pi_{\bar h})$ une représentation de la déformation formelle constante $A[[\bar h]]$ de $A$. Si $H_{\Bbbk}^1(A,\text{End}_{\Bbbk} V) = 0$, alors $(V, \pi_{\bar h})$ est une déformation formelle triviale de $\pi_0$.
\end{theorem}

\begin{proof}
Le cas $s=1$ est l'énoncé du théorème dans \cite{guichardet}. Encore une fois, on vérifie que la démonstration se généralise à un anneau commutatif $\Bbbk$ : en vertu de \eqref{eq_repr} et en supposant $H_{\Bbbk}^1(A,\text{End}_{\Bbbk} V) = 0$, il existe $f_1 \in \text{End}_{\Bbbk} V$ tel que $\pi_1 = d (f_1)$. On conjugue alors $\pi_{\bar h}$ par l'endomorphisme $\text{id}_V + h f_1$. On continue le procédé de la même manière que dans la preuve du théorème \ref{thm_hoch}. \\
Le lemme \ref{lem_defor_end} et le corollaire \ref{cor_hoch2} permettent de faire une induction sur $s$, de la même façon que dans la démonstration du théorème \ref{thm_hoch}.
\end{proof}

\section{Définition des Groupes Quantiques d'Interpolation de Langlands de Rang 1} \label{section_def}
Dans cette section on définit, dans l'esprit de \cite{hernandez}, les groupes quantiques d'interpolation de Langlands de rang 1. Si le principe d'interpolation qui motive la définition des groupes quantiques d'interpolation dans \cite{hernandez} est repris ici, les formules qui l'expriment ne sont toutefois pas les mêmes. Plus précisément, la différence entre les définitions données ici et celles présentes dans \cite{hernandez} ne se situe pas seulement au niveau du langage utilisé. Les formules que l'on découvrira ci-dessous sont de prime abord moins intuitives, mais permettent toutefois de donner une présentation plus concise et peut-être moins compliquée. Notons que l'intérêt n'est pas seulement esthétique puisque cette nouvelle définition permettra, dans les sections suivantes, de démontrer plusieurs propriétés nouvelles des groupes quantiques d'interpolation et de leurs représentations. \\

On sait classifier et décrire explicitement toutes les représentations irréductibles sur $\CC$ de dimension finie $L(n)$ ($n \in \NN$) de $\mathfrak{sl}_2$ : voir par exemple \cite[II.7]{humphreys}. Il existe une base de $L(n)$ dont on note les vecteurs $v_0, v_1, \dots, v_n$, pour laquelle l'action de $\mathfrak{sl}_2$ est donnée par les formules ($0 \leq i \leq n$):
$$ H. v_i \ = \ (n- 2i) \: \! v_i \, , \quad X^+.v_i \ = \ (n-i+1) \: \! v_{i-1} \, , \quad X^-.v_i \ = \ (i+1) \: \! v_{i+1} \, , $$
(on pose $v_{-1} := v_{n+1} := 0$).

$L(n)$ se déforme en une représentation $L^h(n)$ de $U_h(\mathfrak{sl}_2)$ (en tant que $\CC[[h]]$-module $L^h(n) = L(n) [[h]]$) :
$$ H. v_i \ = \ (n- 2i) v_i \, , \quad X^+.v_i \ = \ [n-i+1]_Q \, v_{i-1} \, , \quad X^-.v_i \ = \ [i+1]_Q \, v_{i+1} \, , $$
où $[a]_Q := \frac{Q^a - Q^{-a}}{Q - Q^{-1}}$ pour $a \in \ZZ$ et $Q := \exp h \in \CC[[h]]$ (la série formelle $[a]_Q$ est souvent appelée nombre quantique). \\

Le point de départ dans la définition des groupes quantiques d'interpolation de $\mathfrak{sl}_2$ est de déformer une nouvelle fois, par rapport à un paramètre $h'$, l'action de $H$ et $X^{\pm}$ sur $L^h(n)$, en une action sur le $\CC[[h,h']]$-module $L(n) [[h,h']]$. Contrairement à la précédente déformation par rapport à $h$, ces déformations dépendent d'un paramètre $g \in \NN_{\geq 1}$, de telle sorte que l'on a non pas une déformation mais une famille de déformations. Pour $g$ fixé, les actions déformées des générateurs $H$ et $X^{\pm}$ vont vérifier de nouvelles relations (indépendantes de $n$), déformations des relations de $U_h(\mathfrak{sl}_2)$. C'est ces nouvelles relations qui seront alors posées dans la définition du groupe quantique d'interpolation $U_{h,h'}(\mathfrak{sl}_2,g)$. On notera $L^{h,h'}(n,g)$ l'espace $L(n) [[h,h']]$ considéré comme représentation de l'algèbre $U_{h,h'}(\mathfrak{sl}_2,g)$. \\

On peut remarquer la chose suivante : lors de la quantification de $L(n)$, on a en quelque sorte remplacé les nombres $a$ par les nombres quantiques $[a]_Q$. De manière similaire, lors de la déformation par rapport à $h'$, on va remplacer les nombres quantiques $[a]_Q$ par de nouveaux nombres quantiques $[a]_{QT^{\{a\}}}$ :
\begin{equation}
[a]_{QT^{\{a\}}} \ := \ \frac{\left( Q \, T^{\{a\}} \right)^a \ - \ \left( Q \, T^{\{a\}} \right)^{-a}}{Q \, T^{\{a\}} - Q^{-1} \, T^{-\{a\}}} \, ,
\end{equation}
où $T := \exp h' \in \CC[[h']]$ et $\{a\}$ est un polynôme de Laurent en $Q^a$ que l'on décrira un peu plus loin. \\

De manière plus précise, l'action déformée par rapport à $h$ et $h'$ est décrite par les formules suivantes ($0 \leq i \leq n$):
$$ H. v_i \ = \ (n- 2i) v_i \, , $$
$$ X^+.v_i \ = \ \left( [n-i+1]_{QT^{\{n-i+1\}}} \right) v_{i-1} \, , \quad X^-.v_i \ = \ \left( [i+1]_{QT^{\{i+1\}}} \right) v_{i+1} \, . $$
En particulier, on voit que $X^{\pm} X^{\mp} . v_i = \pi_i^{\pm} \: \! v_i$ avec $\pi_i^{\pm}$ produit de deux nombres quantiques. Avant d'aller plus loin dans les explications, donnons la définition des groupes quantiques d'interpolation $U_{h,h'}(\mathfrak{sl}_2,g)$.

\begin{definition} \label{def_uhh}
Soit $g \geq 1$ un entier. \\
$U_{h,h'}(\mathfrak{sl}_2,g)$ est la $\CC[[h,h']]$-algèbre topologiquement engendrée par $X^+$, $X^-$, $H$, $C$ et définie par les relations
$$ [C, X^{\pm}] \ = \ [C, H] \ = \ 0 \, , \quad \quad [H, X^{\pm}] \ = \ \pm 2 \, X^{\pm} \, , $$

\begin{equation} \label{eq_defx}
X^\pm X^\mp \ = \ \xxp{\pm}{} \, ,
\end{equation}

avec
$$ Q \ := \ \text{exp}(h) \, , \quad  T \ := \ \text{exp} (h') \, , $$
$$ H_e^{\pm} \ := \ \frac{1}{2} \, \Big( \sqrt{C} + e \: \! H \mp e \Big) \, , \ \quad \quad {\{ H_e^{\pm} \}}_{\! \! \: Q} \ := \ P \! \left( Q^{H_e^{\pm}} \right) , \quad \quad e \in \{ -1,1 \} \, , $$
\begin{equation} \label{eq_poly}
P(u) \ := \ \frac{1}{2} \left( u^{g-1} + u^{1-g} \right) \ \prod_{k=1}^{g-1} \, \frac{\varepsilon^k \: \! u - \varepsilon^{-k} \: \! u^{-1}}{\varepsilon^k - \varepsilon^{-k}} \, .
\end{equation}
\end{definition}

Soit $P_g$ le polynôme interpolateur de Lagrange de degré $g-1$ vérifiant
$$ P_g(1) \ = \ 1 \, , \quad P_g(\varepsilon^2) \ = \ \cdots \ = \ P_g(\varepsilon^{2g-2}) \ = \ 0 \, , $$
où $\varepsilon := \exp (i\pi /g) \in \CC$ est une racine $(2g)$-ième primitive de l'unité. \\
Le polynôme de Laurent $P$ défini dans \eqref{eq_poly} peut s'obtenir à partir de $P_g$ par symétrisation (remarquons que $P$ est invariant par les deux transformations $u \mapsto -u$ et $u \mapsto u^{-1}$) :
$$ P(u) \ = \ \frac{1}{2} \Big( P_g(u^2) \ + \ P_g(u^{-2}) \Big) \, . $$
Par conséquent le polynôme $P$ vérifie
\begin{equation} \label{eq_valp}
P(\varepsilon^l) \ = \ 1 \quad \text{si $g$ divise $l$} \, , \quad \ P(\varepsilon^l) \ = \ 0 \quad \text{sinon} \, .
\end{equation}

L'élément $C$ dans $U_{h,h'}(\mathfrak{sl}_2,g)$ doit en quelque sorte être compris comme l'opérateur de multiplication par $(n+1)^2$ sur $L^{h,h'} (n,g)$.
Dans la relation \eqref{eq_defx}, on peut en quelque sorte reconnaître à droite de l'égalité le produit $\pi_i^{\pm}$. En effet l'opérateur $H_e^{\pm}$ sur $L^{h,h'} (n,g)$ est décrit par ($0 \leq i \leq n$) :
$$ H_1^+ . v_i \ = \ n-i \, , \quad H_1^- . v_i \ = \ n-i +1 \, , \quad H_{-1}^+ . v_i \ = \ i+1 \, , \quad H_{-1}^- . v_i \ = \ i \, . $$

La relation \eqref{eq_defx} contient pour ainsi dire la définition des polynômes $\{a\}$ ($a \in \ZZ$) :
\begin{equation}
\{a\} \ := \ P \! \left( Q^a \right) \ = \ \frac{1}{2} \left( Q^{a(g-1)} + Q^{a(1-g)} \right) \ \prod_{k=1}^{g-1} \, \frac{\varepsilon^k \: \! Q^a - \varepsilon^{-k} \: \! Q^{-a}}{\varepsilon^k - \varepsilon^{-k}} \, .
\end{equation}
Ces polynômes $\{a\}$ possèdent la propriété suivante : en spécialisant $Q$ à $\varepsilon$ on obtient d'après \eqref{eq_valp}
$$ \{a\} = 1 \quad \text{si $g$ divise $a$} \quad \text{ et } \quad \{a\}= 0 \quad \text{sinon}. $$

C'est cette propriété des polynômes $\{a \}$ qui permettra de démontrer que $L^{h,h'}(n,g)$ interpole la représentation $L^h(n)$ du groupe quantique $\Uh$ et la représentation Langlands $g$-duale ${}^L \! L^{h'}(n,g)$ de $\Urh{\! gh'}$ : voir la section \ref{section_langlands} et le théorème \ref{thm_reprinter}. \\

Il faut préciser deux points dans la définition de $\Uhh$.
\begin{itemize}
\item[\textbullet] On peut considérer
$$ \xxn{e}{\pm}{}{} \quad \text{ et } \quad \xxd{e}{\pm}{}{} $$
comme des séries formelles dans $\CC \left[H,\sqrt{C} \right] [[h,h']]$, où $H$ et $\sqrt C$ désignent ici les variables de l'algèbre de polynômes $\CC \left[H,\sqrt{C} \right]$. Les deux séries formelles sont divisibles par $h + h' {\{ H_e^{\pm} \}}_{\! Q}$. Le quotient
$$ \xxf{e}{\pm}{}{} $$
est alors défini dans $\CC \left[H,\sqrt{C} \right] [[h,h']]$ comme le quotient des deux séries formelles précédentes chacune divisée par $h + h' {\{ H_e^{\pm} \}}_{\! Q}$ (de telle sorte que la série au dénominateur devient alors inversible).
\item[\textbullet] En considérant le produit
$$ \Pi \ := \ \xxp{\pm}{} $$
comme un élément de l'algèbre $\CC \left[H,\sqrt{C} \right] [[h,h']]$, on vérifie que $\Pi$ est invariant par l'unique automorphisme de la $\CC[[h,h']]$-algèbre $\CC \left[H,\sqrt{C} \right] [[h,h']]$ qui envoie $H$ sur $H$ et $\sqrt C$ sur $- \sqrt C$ (car le polynôme $P$ défini dans \eqref{eq_poly} est invariant par la transformation $u \mapsto u^{-1}$). Cela prouve que $\Pi$ appartient en fait à la sous-algèbre $\CC \left[H,\left( \sqrt{C} \right)^2 \right] [[h,h']]$.
\end{itemize}

Dans la relation \eqref{eq_defx}, le terme à droite de l'égalité est l'image de $\Pi$ par le morphisme de $\CC[[h,h']]$-algèbre qui à $H$ associe $H$ et à $\left( \sqrt{C} \right)^2$ l'élément $C$. \\
$\sqrt C$ ne désigne donc pas un élément de l'algèbre $U_{h,h'}(\mathfrak{sl}_2,g)$. On utilisera souvent cette écriture dans la suite en omettant toutefois de faire les vérifications semblables à celles faites précédemment. 

\begin{rem} \label{rem_hernandezgen}
Dans \cite{hernandez}, les groupes quantiques d'interpolation nécessitent d'autres générateurs que $X^{\pm}$, $H$ et $C$ (on verra dans la section suivante que l'on peut en fait se passer de $C$ pour définir $\Uhh$). De plus, il est nécessaire de considérer des puissances de certains des générateurs par d'autres générateurs, ce à quoi on peut donner un sens quitte à rajouter au moins autant de nouveaux générateurs qu'il n'y a de telles puissances. Ces algèbres sont en quelque sorte beaucoup plus grosses que les algèbres $\Uhh$, donc moins maniables, et plusieurs des résultats qui suivront dans les sections suivantes n'ont pas d'analogues dans \cite{hernandez}.
\end{rem}

$U(\mathfrak{sl}_2)$ désignera l'algèbre enveloppante de $\mathfrak{sl}_2$ : c'est la $\CC$-algèbre engendrée par $\bar{X}^{\pm}$, $\bar{H}$, et définie par les relations
$$ [\bar{H}, \bar{X}^{\pm}] \ = \ \pm 2 \, \bar{X}^{\pm} \, , \quad [\bar{X}^+, \bar{X}^-] \ = \ \bar{H} \, . $$
Le centre de $U(\mathfrak{sl}_2)$ est l'algèbre des polynômes en l'élément de Casimir
\begin{equation} \label{eq_casimir}
\bar{C} \ := \ (\bar{H} + 1)^2 \ + \ 4 \, \bar{X}^- \bar{X}^+ \ = \ (\bar{H} - 1)^2 \ + \ 4 \, \bar{X}^+ \bar{X}^- \, ,
\end{equation}
voir par exemple \cite[p. 304]{knapp}. \\

Rappelons par ailleurs la définition du groupe quantique $U_{\: \! h}(\mathfrak{sl}_2)$ (qui est une déformation formelle triviale de $U(\mathfrak{sl}_2)$ : voir par exemple \cite[6.4]{chari}). \\
Il s'agit de la $\CC[[h]]$-algèbre topologiquement engendrée par les éléments $\tilde{X}^{\pm}$, $\tilde H$, et définie par les relations
$$ [\tilde{H}, \tilde{X}^{\pm}] \ = \ \pm 2 \, \tilde{X}^{\pm} \, , \quad [\tilde{X}^+, \tilde{X}^-] \ = \ \frac{\text{sinh} (h \: \! \tilde{H})}{\text{sinh}(h)} \, . $$

On notera par ailleurs $U_{r h}(\mathfrak{sl}_2)$ ($r \in \NN_{\geq 1}$) la $\CC[[h]]$-algèbre topologiquement engendrée par les éléments $\tilde{X}^{\pm}$, $\tilde{H}$, et définie par les relations
$$ [\tilde{H}, \tilde{X}^{\pm}] \ = \ \pm 2 \, \tilde{X}^{\pm} \, , \quad [\tilde{X}^+, \tilde{X}^-] \ = \ \frac{\text{sinh} (r h \: \! \tilde{H})}{\text{sinh}(r h)} \, . $$
En d'autres mots, $U_{r h}(\mathfrak{sl}_2)$ est simplement obtenue à partir de $U_{\: \! h}(\mathfrak{sl}_2)$ par renormalisation du paramètre $h$.

\begin{exemple} \label{ex_g1}
Dans la définition \ref{def_uhh}, considérons le cas où $g = 1$. On a $P = 1$ pour le polynôme d'interpolation définie dans \eqref{eq_poly}. Par suite $U_{h,h'}(\mathfrak{sl}_2,1)$ est la $\CC[[h,h']]$-algèbre topologiquement engendrée par $X^+$, $X^-$, $H$, $C$ et définie par les relations
$$ [C, X^{\pm}] \ = \ [C, H] \ = \ 0 \, , $$
$$ [H, X^{\pm}] \ = \ \pm 2 \, X^{\pm} \, , $$
$$ X^\pm X^\mp \ = \ \prod_{e = 1,-1} \frac{\left( Q \, T \right)^{H_e^{\pm}} - \ \left( Q \, T \right)^{-H_e^{\pm}}}{Q \, T \ - \ Q^{-1} \, T^{-1} } \, , $$
avec
$$ H_e^{\pm} \ := \ \frac{1}{2} \, \Big( \sqrt{C} + e \: \! H \mp e \Big) \, , \ \quad  \text{où } \ e \in \{ -1,1 \} \, . $$
La relation suivante est vérifiée dans $U_{h,h'}(\mathfrak{sl}_2,1)$ :
$$ [X^+, X^-] \ = \ \frac{\left( QT \right)^H - \left( QT \right)^{-H}}{QT - Q^{-1} T^{-1}} \, . $$
Par conséquent, $U_{h,h'}(\mathfrak{sl}_2,1)$ contient naturellement un quotient de la $\CC[[h'']]$-algèbre $U_{h''}(\mathfrak{sl}_2)$ (toute $\CC[[h,h']]$-algèbre, et en particulier $U_{h,h'}(\mathfrak{sl}_2,1)$, est naturellement munie d'une structure de $\CC[[h'']]$-algèbre où $h''$ agit par $h + h'$). D'après la proposition \ref{prop_basehh} énoncée un peu plus loin, on voit que $U_{h''}(\mathfrak{sl}_2)$ est en fait une sous-algèbre de $U_{h,h'}(\mathfrak{sl}_2,1)$. En notant $h^{(3)} = h - h'$, $U_{h,h'}(\mathfrak{sl}_2,1)$ apparaît comme la déformation constante de $U_{h''}(\mathfrak{sl}_2)$ selon $h^{(3)}$. \\
Cette situation est particulière à $g = 1$. Pour $g \geq 2$, $U_{h,h'}(\mathfrak{sl}_2,g)$ n'est plus une déformation formelle constante selon $h^{(3)}$ du groupe quantique usuel. En particulier, on remarque que la symétrie entre $h$ et $h'$, observée pour $g=1$, est brisée pour $g \geq 2$.
\end{exemple}

\begin{rem}
Les groupes quantiques $\Uhh$ pour $g = 1,2$ et $3$ correspondent aux groupes quantiques d'interpolation appelés élémentaires dans \cite{hernandez}, et respectivement notés $U_{q,t}(A_1)$, $U_{q,t}(B_1)$ et $U_{q,t}(G_1)$. \\
Ces groupes quantiques d'interpolation élémentaires permettent de donner une définition dans \cite{hernandez} d'un groupe quantique d'interpolation de Langlands associé à une algèbre de Lie simple de dimension finie. Dans cet article, la définition donnée est valable quelque soit $g$, ce qui permettra d'associer des groupes d'interpolation de Langlands à toute algèbre de Kac-Moody (symétrisable).
\end{rem}

\section{Propriétés des Groupes Quantiques d'Interpolation} \label{section_prop}
On étudie ici les groupes quantiques d'interpolation de Langlands $\Uhh$ définis précédemment. La section est divisée en cinq sous-sections. \\
Dans la première on montre que les groupes quantiques d'interpolation s'avèrent être des doubles déformations formelles triviales de $U(\mathfrak{sl}_2)$, et plus encore, des déformations triviales selon $h$ (ou $h'$) du groupe quantique usuel associé à $\mathfrak{sl}_2$ : voir le théorème \ref{thm_uhhdefor}. Il s'agit là d'un fait nouveau et d'un point clé de l'article. \\
Les deux sous-sections qui suivent donnent différentes propriétés des groupes quantiques d'interpolation, obtenues grâce aux structures de double déformation de ceux-ci. La plupart, comme expliqué dans la remarque \ref{rem_hernandezgen}, n'ont pas d'analogues dans \cite{hernandez}. \\
On prouve notamment l'existence d'une décomposition triangulaire de $\Uhh$ (voir la proposition \ref{prop_triang}), qui sera un outil efficace, par exemple pour l'étude des représentations. \\
Dans la quatrième sous-section, on donne quelques propriétés de symétrie de $\Uhh$. Enfin dans la dernière sous-section, on définit rigoureusement ce que signifie spécialiser $\Uhh$ en $Q = \varepsilon$, et on établit un lemme crucial pour notre étude : voir le lemme \ref{lem_fond} et la remarque \ref{rem_fond}.

\subsection{Doubles déformations}
La proposition un peu technique qui suit est fondamentale, elle contient en quelque sorte le noyau dur de la preuve que $\Uhh$ admet des structures de déformations de $U(\mathfrak{sl}_2)$ et du groupe quantique usuel : voir le théorème \ref{thm_uhhdefor}.

\begin{proposition} \label{prop_uconst}
Il existe un isomorphisme de $\CC[[h,h']]$-algèbre
$$ \psi \, : \ U_{h,h'}(\mathfrak{sl}_2,g) \ \stackrel{\sim}{\longrightarrow} \, \ \Uc[[h,h']] \, , $$
où $\Uc[[h,h']]$ désigne la déformation formelle constante de $\Uc$, et tel que
$$ \psi(H) \ = \ \bar{H} \, , \quad \psi(C) \ = \ \bar{C} \, , \quad \psi(X^-) \ = \ \bar{X}^- \,  , $$
\begin{equation} \label{eq_uconst}
\psi(X^+) \ = \ \left( \prod_{e = 1, -1} \ \frac{\Big( Q \, T^{{\{ \bar{H}_e^+ \}}_Q} \Big)^{\bar{H}_e^+} - \ \Big( Q \, T^{{\{ \bar{H}_e^+ \}}_Q} \Big)^{-\bar{H}_e^+}}{\bar{H}_e^+ \Big( Q \, T^{{\{ \bar{H}_e^+ \}}_Q} \ - \ Q^{-1} \, T^{-{\{ \bar{H}_e^+ \}}_Q} \Big)} \right) \bar{X}^+ \, ,
\end{equation}
\end{proposition}

Avant de donner la démonstration de cette proposition, commençons par expliquer le terme à droite de l'égalité \eqref{eq_uconst}. \\
$\bar{H}_e^{\pm}$ est défini comme dans la définition \ref{def_uhh} :
$$ \bar{H}_e^{\pm} \ := \ \frac{1}{2} \, \Big( \sqrt{\bar{C}} + e \: \! \bar{H} \mp e \Big) \, , \ \quad  \ e \in \{ -1,1 \} \, . $$
Dans $\CC[H,\sqrt{\bar C}][[h,h']]$, la série formelle
\begin{equation*}
\Big( Q \, T^{{\{ \bar{H}_e^+ \}}_Q} \Big)^{\bar{H}_e^+} - \ \Big( Q \, T^{{\{ \bar{H}_e^+ \}}_Q} \Big)^{-\bar{H}_e^+} \ = \ 2 \: \! \sinh \! \Bigg[ \! \left( h + h' { \{ \bar{H}^+_e \} }_{\! Q} \right) \! \bar{H}^+_e \Bigg]
\end{equation*}
est divisible par $\left( h + h' { \{ \bar{H}^+_e \} }_{\! Q} \right) \! \bar{H}^+_e$. La série
$$ Q \, T^{{\{ \bar{H}_e^+ \}}_Q} \ - \ Q^{-1} \, T^{-{\{ \bar{H}_e^+ \}}_Q} \ = \ 2 \: \! \sinh \! \left( h + h' { \{ \bar{H}^+_e \} }_{\! Q} \right) $$
est quant à elle divisible par $h + h' { \{ \bar{H}^+_e \} }_{\! Q}$ dans $\CC[H,\sqrt{\bar C}][[h,h']]$. Le quotient
$$ \bar{F} \ := \ \left( \prod_{e = 1, -1} \ \frac{\Big( Q \, T^{{\{ \bar{H}_e^+ \}}_Q} \Big)^{\bar{H}_e^+} - \ \Big( Q \, T^{{\{ \bar{H}_e^+ \}}_Q} \Big)^{-\bar{H}_e^+}}{\bar{H}_e^+ \Big( Q \, T^{{\{ \bar{H}_e^+ \}}_Q} \ - \ Q^{-1} \, T^{-{\{ \bar{H}_e^+ \}}_Q} \Big)} \right) $$
a donc un sens dans $\CC[H,\sqrt{\bar C}][[h,h']]$ une fois les divisions faites. De la même manière que dans la définition \ref{def_uhh}, on vérifie que $\bar F$ appartient en fait à la sous-algèbre $\CC[H,(\sqrt{\bar C})^2][[h,h']]$. \\
Le terme à droite de l'égalité dans \eqref{eq_uconst} est défini comme l'image de $\bar F$ par l'unique morphisme de $\CC[[h,h']]$-algèbre qui envoie $\bar H$ sur $\bar H$ et $(\sqrt{\bar C})^2$ sur $\bar C$.

\begin{proof}
Pour prouver l'existence du morphisme $\psi$, il s'agit, d'après le point 4) de la proposition \ref{prop_props}, de voir que
$$ \Big( \psi(X^+), \psi(X^-), \psi(H), \psi(C) \Big) $$
vérifient les relations définissant $U_{ \: \! h,h'} (\mathfrak{sl}_2)$.

\begin{itemize}
\item[\textbullet] Puisque $\psi$ envoie $C$ sur l'élément central $\bar C$, les relations suivantes sont immédiatement vérifiées :
$$ \big[ \psi(C), \psi(X^{\pm}) \big] \ = \ \big[ \psi(C), \psi(H) \big] \ = \ 0 \, . $$

\item[\textbullet] $$ \big[ \psi(H), \psi(X^-) \big] \ = \ \big[ \bar{H}, \bar{X}^- \big] \ = \ - \: \! 2 \: \! \bar{X}^- \ = \ \ - \: \! 2 \: \! \psi(X^-) \, . $$

\item[\textbullet] L'élément
$$ \bar \Pi \ := \ \prod_{e = 1, -1} \ \frac{\Big( Q \, T^{{\{ \bar{H}_e^+ \}}_Q} \Big)^{\bar{H}_e^+} - \ \Big( Q \, T^{{\{ \bar{H}_e^+ \}}_Q} \Big)^{-\bar{H}_e^+}}{\bar{H}_e^+ \Big( Q \, T^{{\{ \bar{H}_e^+ \}}_Q} \ - \ Q^{-1} \, T^{-{\{ \bar{H}_e^+ \}}_Q} \Big)} $$
appartient à la sous-$\CC[[h,h']]$-algèbre topologiquement engendrée par $\bar H$ et $\bar C$, donc commute avec $\bar H$. Par suite,
$$ \big[ \psi(H), \psi(X^+) \big] \ = \ \big[ \bar{H}, \bar \Pi \: \! \bar{X}^+ \big] \ = \ \bar \Pi \big[ \bar{H},\bar{X}^+ \big] \ = \ 2 \: \! \bar \Pi \: \! \bar{X}^+ \ = \ 2 \: \! \psi(X^+) \, . $$

\item[\textbullet] Dans $U_h(\mathfrak{sl}_2)$, on a $\bar{X}^+ \bar{X}^- = \frac{1}{4} \bar{C} - \frac{1}{4} (\bar{H} -1)^2 = \bar{H}_1^+ \: \! \bar{H}_{-1}^+$, donc
\begin{eqnarray*} \psi(X^+) \: \! \psi(X^-) & =  & \prod_{e = 1, -1} \ \frac{\Big( Q \, T^{{\{ \bar{H}_e^+ \}}_Q} \Big)^{\bar{H}_e^+} - \ \Big( Q \, T^{{\{ \bar{H}_e^+ \}}_Q} \Big)^{-\bar{H}_e^+}}{Q \, T^{{\{ \bar{H}_e^+ \}}_Q} \ - \ Q^{-1} \, T^{-{\{ \bar{H}_e^+ \}}_Q} \Big)} \\
& = & \psi \left( \xxp{+}{} \right) \, .
\end{eqnarray*}

\item[\textbullet] Toujours dans $U_h(\mathfrak{sl}_2)$, on a $\bar{X}^- \bar{H}_e^+ \ = \ \bar{H}_e^- \bar{X}^-$, donc
\begin{equation*}
\bar{X}^- \left( \prod_{e = 1, -1} \ \frac{\Big( Q \, T^{{\{ \bar{H}_e^+ \}}_Q} \Big)^{\bar{H}_e^+} - \ \Big( Q \, T^{{\{ \bar{H}_e^+ \}}_Q} \Big)^{-\bar{H}_e^+}}{\bar{H}_e^+ \Big( Q \, T^{{\{ \bar{H}_e^+ \}}_Q} \ - \ Q^{-1} \, T^{-{\{ \bar{H}_e^+ \}}_Q} \Big)} \right) \ = \ \left( \prod_{e = 1, -1} \ \frac{\Big( Q \, T^{{\{ \bar{H}_e^- \}}_Q} \Big)^{\bar{H}_e^-} - \ \Big( Q \, T^{{\{ \bar{H}_e^- \}}_Q} \Big)^{-\bar{H}_e^-}}{\bar{H}_e^- \Big( Q \, T^{{\{ \bar{H}_e^- \}}_Q} \ - \ Q^{-1} \, T^{-{\{ \bar{H}_e^- \}}_Q} \Big)} \right) \bar{X}^- \, . 
\end{equation*}
Par ailleurs, $\bar{X}^- \bar{X}^+ = \frac{1}{4} \bar{C} - \frac{1}{4} (\bar{H} +1)^2 = \bar{H}_1^- \: \! \bar{H}_{-1}^-$, donc
\begin{eqnarray*}
\psi(X^-) \: \! \psi(X^+) & = & \prod_{e = 1, -1} \ \frac{\Big( Q \, T^{{\{ \bar{H}_e^- \}}_Q} \Big)^{\bar{H}_e^-} - \ \Big( Q \, T^{{\{ \bar{H}_e^- \}}_Q} \Big)^{-\bar{H}_e^-}}{Q \, T^{{\{ \bar{H}_e^- \}}_Q} \ - \ Q^{-1} \, T^{-{\{ \bar{H}_e^- \}}_Q}} \\
& = & \psi \left( \xxp{-}{} \right) \, .
\end{eqnarray*}
\end{itemize}
 
On peut définir un inverse $\psi_1$ de $\psi$ en posant
$$ \psi_1(\bar{H}) \ = \ H \, , \quad \psi_1(\bar{X}^-) \ = \ X^- \, , $$
\begin{equation} \label{eq_uconst2}
\psi_1(\bar{X}^+) \ = \ \left( \prod_{e = 1, -1} \ \frac{H_e^+ \Big( Q \, T^{{\{ H_e^+ \}}_Q} \ - \ Q^{-1} \, T^{-{\{ H_e^+ \}}_Q} \Big)}{\Big( Q \, T^{{\{ H_e^+ \}}_Q} \Big)^{H_e^+} - \ \Big( Q \, T^{{\{ H_e^+ \}}_Q} \Big)^{-H_e^+}} \right) X^+ \, .
\end{equation}
En reprenant la discussion précédant cette démonstration, on voit que le terme à droite de l'égalité \eqref{eq_uconst2} est bien défini : il faut remarquer que la série formelle
$$ \xxn{e}{+}{}{} \ = \ 2 \: \! \sinh \! \Bigg[ \! \left( h + h' { \{ H^+_e \} }_{\! Q} \right) \! H^+_e \Bigg] \, , $$
une fois divisée par $\left( h + h' { \{ H^+_e \} }_{\! Q} \right) \! H^+_e$, est inversible dans $\CC[H, \sqrt C][[h,h']]$. \\
On prouve grâce à des arguments similaires à ceux donnés pour $\psi$ que $\psi_1$ est un morphisme de $\CC[[h,h']]$-algèbre. En particulier on obtient que
$$ \psi_1(\bar{X}^{\pm}) \: \! \psi_1(\bar{X}^{\mp}) \ = \ \frac{1}{4} C \ - \ \frac{1}{4} (H \mp 1)^2 \, , $$
ce qui implique que $\psi_1(\bar{C}) = C$ et par suite que $(\psi_1 \circ \psi)(C) = C$ et $(\psi \circ \psi_1) (\bar C) = C$. Les autres identités à vérifier, pour montrer que $\psi$ et $\psi_1$ sont inverses l'un de l'autre, sont évidentes.
\end{proof}

La proposition \ref{prop_uconst} montre en particulier que la $\CC[[h,h']]$-algèbre $U_{h,h'}(\mathfrak{sl}_2,g)$ est isomorphe à la déformation formelle constante de $U(\mathfrak{sl}_2)$ selon $h$ et $h'$. \\

Il existe une autre façon, plus naturelle, d'identifier les $\CC[[h,h']]$-modules $U_{h,h'}(\mathfrak{sl}_2,g)$ et $U(\mathfrak{sl}_2)[[h,h']]$, et par suite une façon plus naturelle de voir $U_{h,h'}(\mathfrak{sl}_2,g)$ comme une déformation formelle de $U(\mathfrak{sl}_2)$ selon $h$ et $h'$. On identifie la base $\Big( (\bar{X}^-)^a \bar{H}^b (\bar{X}^+)^c \Big)_{a,b,c \in \NN}$ de $U(\mathfrak{sl}_2)$ avec la famille $\Big( (X^-)^a H^b (X^+)^c \Big)_{a, b, c \in \NN}$ de $U_{h,h'}(\mathfrak{sl}_2,g)$, qu'on vérifie en effet être une base topologique dans la proposition suivante.

\begin{proposition} \label{prop_basehh}
Soit $g \in \NN_{\geq 1}$. \\
$U_{\: \! h,h'}(\mathfrak{sl}_2,g)$ admet pour base topologique la famille $\left( (X^-)^a H^b (X^+)^c \right)_{a, b, c \in \NN} $.
\end{proposition}

\begin{proof}
On utilise l'isomorphisme $\psi$ de la proposition \ref{prop_uconst} et le théorème de Poincaré-Birkhoff-Witt pour $\mathfrak{sl}_2$, en remarquant que
\begin{equation} \label{eq_psimod}
\psi(H) \ \equiv \ \bar{H} \ \text{ mod }(h,h') \, , \quad \ \psi(X^\pm) \ \equiv \bar{X}^{\pm} \ \text{ mod }(h,h') \, .
\end{equation}
\end{proof}

Comme indiqué plus haut, l'existence de cette base topologique fournit un isomorphisme $\CC[[h,h']]$-linéaire naturel $B^{h,h'} : \Uc[[h,h']] \stackrel{\sim}{\to} \Uhh$, puis, par transport, une structure de $\CC[[h,h']]$-algèbre sur $U(\mathfrak{sl}_2)[[h,h']]$. \\
En considérant ci-dessous $\Uc[[h,h']]$ à gauche muni de cette structure et $\Uc[[h,h']]$ à droite comme la déformation constante de $\Uc$,
$$ \psi \circ B^{h,h'} \, : \ \Uc[[h,h']] \ \to \ \Uc[[h,h']] $$
est alors un isomorphisme de $\CC[[h,h']]$-algèbre. En utilisant \eqref{eq_psimod} on en conclut que la structure transportée par $B^{h,h'}$ sur $\Uc[[h,h']]$ est une déformation formelle de $U(\mathfrak{sl}_2)$ par rapport à $h$ et $h'$, qu'on notera par abus de langage $\Uhh$. \\

$U_h(\mathfrak{sl}_2)$ admet comme base topologique la famille de monômes $(\tilde{X}^-)^a \tilde{H}^b (\tilde{X}^+)^c$ avec $a,b,c \in \NN$ : voir par exemple \cite[6.4]{chari}. De même que précédemment, par transport on définit sur $U(\mathfrak{sl}_2)[[h]]$ une structure de déformation formelle de $U(\mathfrak{sl}_2)$ par rapport à $h$, qu'on notera par abus de langage $\Uh$. \\

Dans le quotient $U_{h,h'}(\mathfrak{sl}_2,g) / (h'=0)$ les relations suivantes sont vérifiées : $[H, X^{\pm}] = \pm 2 \: \! X^{\pm}$ et
\begin{eqnarray*}
[X^+, X^-] & = &  X^+ X^- \ - \ X^- X^+ \\
& = & \frac{Q^{\sqrt{C}} + Q^{-\sqrt{C}}}{(Q - Q^{-1})^2} \ - \ \frac{Q^{H - 1} + Q^{-H+1}}{(Q - Q^{-1})^2} \ - \ \frac{Q^{\sqrt{C}} + Q^{-\sqrt{C}}}{(Q - Q^{-1})^2} \ + \ \frac{Q^{H + 1} + Q^{-H-1}}{(Q - Q^{-1})^2} \\
& = & \frac{Q^H - Q^{-H}}{Q - Q^{-1}} \, .
\end{eqnarray*}

Cela prouve l'existence d'un morphisme de $\CC[[h]]$-algèbre
$$ \Uh = \Uc[[h]] \ \longrightarrow \ \Uhh / (h'=0) = \Uc[[h,h']] / (h'=0) = \Uc[[h]] $$
égal à l'idendité. En d'autres mots $\Uhh = \Uc[[h,h']] = \Uc[[h]][[h']]$ est une déformation formelle selon $h'$ de $\Uh = \Uc[[h]]$. \\

De manière similaire, on voit que $\Uhh$ est une déformation formelle de $\Urh{h'}$ selon $h$ : il faut remarquer que dans le quotient $U_{h,h'}(\mathfrak{sl}_2,g) / (h=0)$ on a $P(Q^{H_e^{\pm}}) = 1$. \\

L'isomorphisme $\psi$ de la proposition \ref{prop_uconst} induit un isomorphisme de $\CC[[h]]$-algèbre
$$ \varphi \ := \ \lim_{h' \to 0} \psi \, : \ U_h(\mathfrak{sl}_2) \ \stackrel{\sim}{\longrightarrow} \ U(\mathfrak{sl}_2)[[h]] \, , $$
qui n'est autre que l'isomorphisme de \cite[prop. 6.4.6]{chari}. Pour le voir, il suffit de calculer $\varphi$ sur les générateurs $\tilde{X}^{\pm}$ et $\tilde H$, faisons-le pour $\tilde{X}^+$ (c'est évident pour les autres) :
\begin{eqnarray*}
\varphi(\tilde{X}^+) & = & \left( \prod_{e=1,-1} \frac{Q^{\bar{H}_e^+} - Q^{- \bar{H}_e^+}}{\bar{H}_e^+ \left( Q - Q^{-1} \right)} \right) \bar{X}^+ \\
& = & 4 \left( \frac{Q^{\sqrt{\bar{C}}} + Q^{-\sqrt{\bar{C}}} - Q^{\bar{H} - 1} - Q^{-\bar{H} +1}}{\left( \bar{C} - (\bar{H}-1)^2 \right) \left( Q - Q^{-1} \right)^2} \right) \bar{X}^+ \\
& = & 2 \left( \frac{\cosh h(\bar{H} -1) - \: \! \cosh h \sqrt{\bar{C}}}{\left( (\bar{H}-1)^2 - \bar{C} \right) \sinh^2 h} \right) \bar{X}^+ \, .
\end{eqnarray*}

On pose $\phi$ l'unique isomorphisme de $\CC[[h,h']]$-algèbre tel que le diagramme suivant est commutatif :
$$ \xymatrix{
&& U_h(\mathfrak{sl}_2) [[h']] \ar[rrd]_-{\sim}^-{\mathcal{Q}^{h'} (\varphi)} && \\
U_{h,h'} (\mathfrak{sl}_2,g) \ar[rru]^-{\phi}_-{\sim} \ar[rrrr]^-{\psi}_-{\sim} && && U(\mathfrak{sl}_2) [[h,h']]
} $$

De la même manière, on a l'existence d'un isomorphisme de $\CC[[h']]$-algèbre
$$ \varphi' \ := \ \lim_{h \to 0} \psi \, : \ U_{h'}(\mathfrak{sl}_2) \ \stackrel{\sim}{\longrightarrow} \ U(\mathfrak{sl}_2)[[h']] \, , $$
et d'un unique isomorphisme $\phi'$ de $\CC[[h,h']]$-algèbre, de telle sorte que le diagramme précédent se complète par symétrie en le diagramme commutatif suivant :
\begin{eqnarray} \label{eq_diag1}
\xymatrix{
&& U_h(\mathfrak{sl}_2) [[h']] \ar[rrd]_-{\sim}^-{\mathcal{Q}^{h'} (\varphi)} && \\
U_{h,h'} (\mathfrak{sl}_2,g) \ar[rru]^-{\phi}_-{\sim} \ar[rrd]_-{\phi'}^-{\sim} \ar[rrrr]^-{\psi}_-{\sim} && && U(\mathfrak{sl}_2) [[h,h']] \\
&& U_{h'}(\mathfrak{sl}_2) [[h]] \ar[rru]^-{\sim}_-{\mathcal{Q}^{h} (\varphi')} && 
} \end{eqnarray}

On résume la discussion précédente dans le théorème \ref{thm_uhhdefor} suivant, qu'on illustre par le diagramme commutatif
$$ \xymatrix{
&& \Uhh \ar[dll]_-{\lim_{h' \to 0} \ } \ar[drr]^-{\lim_{h \to 0}} \ar[dd]^-{\lim_{h,h' \to 0}} && \\
\Uh \ar[drr]_-{\lim_{h \to 0}} && && \Urh{\: \! h'} \ar[dll]^-{\ \lim_{h' \to 0}} \\
&& \Uc &&
} $$

\begin{theorem} \label{thm_uhhdefor}
\begin{itemize}
\item[1)] $U_{\: \! h,h'}(\mathfrak{sl}_2,g)$ est une déformation formelle triviale de $U(\mathfrak{sl}_2)$ selon $h, h'$.
\item[2)] $U_{\: \! h,h'}(\mathfrak{sl}_2,g)$ est une déformation formelle triviale de $U_h(\mathfrak{sl}_2)$ selon $h'$.
\item[3)] $U_{\: \! h,h'}(\mathfrak{sl}_2,g)$ est une déformation formelle triviale de $U_{h'}(\mathfrak{sl}_2)$ selon $h$.
\end{itemize}
\end{theorem}

\begin{rem}
La cohomologie de Hochschild de $U(\mathfrak{sl}_2)$ est nulle en degré 2 : voir par exemple \cite[II.11]{guichardet2}. le théorème \ref{thm_hoch} et le lemme \ref{lem_hoch} donnent donc une preuve que les déformations formelles précédentes sont triviales. La proposition \ref{prop_uconst} a toutefois l'avantage de fournir explicitement des isomorphismes avec les déformations constantes.
\end{rem}

\begin{rem}
Le théorème \ref{thm_uhhdefor} signifie en particulier que les groupes quantiques d'interpolation $\Uhh$ ne sont en fait pas plus gros que l'algèbre enveloppante $U(\mathfrak{sl}_2)$ (seul l'espace des scalaires est grossi). Par ailleurs, il permet de formuler le fait qu'ils sont des objets comparables à $U(\mathfrak{sl}_2)$.
\end{rem}

\subsection{Corollaires}
Ici on donne des corollaires des résultats de la section précédente. Ces résultats sont nouveaux. \\
Le premier est une conséquence immédiate de la proposition \ref{prop_basehh}. Il précise que $\Uhh$ est pour ainsi dire défini à partir de seulement trois éléments $X^-$, $H$ et $X^+$.

\begin{corollary} \label{cor_xheng}
$U_{h,h'}(\mathfrak{sl}_2,g)$ est topologiquement engendrée par les éléments $X^{\pm}$ et $H$.
\end{corollary}

Comme conséquences immédiates de la proposition \ref{prop_uconst}, on a les deux corollaires suivants.
\begin{corollary}
$U_{\: \! h,h'}(\mathfrak{sl}_2,g)$ est intègre.
\end{corollary}

\begin{corollary} \label{cor_eng}
Le centre de $U_{h,h'}(\mathfrak{sl}_2,g)$ est topologiquement engendré par $C$ et isomorphe à $\CC[C] [[h,h']]$.
\end{corollary}

On donne maintenant un résultat de décomposition triangulaire, analogue à ceux que l'on connaît pour $\Uc$ et $\Uh$. \\
On note $U^-_{h,h'}(\mathfrak{sl}_2,g)$, $U^0_{h,h'}(\mathfrak{sl}_2,g)$ et $U^+_{h,h'}(\mathfrak{sl}_2,g)$ les sous-$\CC[[h,h']]$-algèbres de $U_{h,h'}(\mathfrak{sl}_2,g)$, topologiquement engendrées respectivement par $X^-$, $H$ et $X^+$. \\
D'après la proposition \ref{prop_basehh}, elles sont, en tant que $\CC[[h,h']]$-algèbres, canoniquement isomorphes respectivement à $\CC[X^-][[h,h']]$, $\CC[H][[h,h']]$ et $\CC[X^+][[h,h']]$.

\begin{proposition} \label{prop_triang}
$U_{\: \! h,h'}(\mathfrak{sl}_2,g)$ admet une décomposition triangulaire :
\begin{equation} \label{eq_triang}
U_{h,h'}(\mathfrak{sl}_2,g) \ \simeq \ U^-_{h,h'}(\mathfrak{sl}_2,g) \ \tilde{\otimes}_{\CC[[h,h']]} \, \ U^0_{h,h'}(\mathfrak{sl}_2,g) \ \tilde{\otimes}_{\CC[[h,h']]} \ \, U^+_{h,h'}(\mathfrak{sl}_2,g) \, ,
\end{equation}
où l'isomorphisme de $\CC[[h,h']]$-modules est défini par la multiplication de $U_{h,h'}(\mathfrak{sl}_2,g)$ :
$$ x^- \otimes x^0 \otimes x^+ \ \mapsto \ x^- \cdot x^0 \cdot x^+ \, , \quad \text{ avec } \ x^s \in U^s_{h,h'}(\mathfrak{sl}_2,g), \ s \in \{-,0,+ \} \, . $$
\end{proposition}

\begin{proof}
C'est une conséquence de la proposition \ref{prop_basehh}.
\end{proof}

\subsection{Autre présentation}
On donne ici une autre présentation des algèbres $\Uhh$. Elle présente l'avantage de ne faire appel qu'aux générateurs $X^{\pm}$ et $H$. Les résultats ici sont nouveaux. \\

$U_{\: \! h,h'}(\mathfrak{sl}_2,g)$ est engendrée par $X^{\pm}$, $H$ et $C$. Le corollaire \ref{cor_xheng} signifie donc que $C$ peut s'écrire en fonction de $X^{\pm}$ et $H$. Plus précisément, on a le résultat suivant.

\begin{defiprop} \label{defiprop_cxx}
Soit $\CC \langle X^{\pm}, H \rangle^0$ le sous-$\CC$-espace vectoriel de $\CC \langle X^{\pm}, H \rangle$ engendré par les monômes de la forme
$$ (X^-)^a H^b (X^+)^a \, , \quad a, b \in \NN \, . $$

\begin{itemize}
\item[1)] Il existe une unique famille d'éléments $C_{n,m} \in \CC \langle X^{\pm}, H \rangle^0$, avec $n,m \in \NN$, telle que dans $U_{\: \! h,h'}(\mathfrak{sl}_2,g)$
$$ C \ = \ \sum_{n,m \, \in \, \NN} C_{n,m} \, h^n (h')^m \, . $$

\item[2)] Il existe une unique famille d'éléments $\alpha_{n,m} \in \CC \langle X^{\pm}, H \rangle^0$, avec $n,m \in \NN$, telle que dans $U_{\: \! h,h'}(\mathfrak{sl}_2,g)$
$$ [X^+, X^-] \ = \ \sum_{n,m \, \in \, \NN} \alpha_{n,m} \, h^n (h')^m \, . $$
\end{itemize}
\end{defiprop}

\begin{proof}
La proposition \ref{prop_basehh} implique que le sous-$\CC[[h,h']]$-module de $\Uhh$ topologiquement engendré par les monômes de la forme
$$ (X^-)^a H^b (X^+)^a \, , \quad a, b \in \NN $$
est égal à la sous-$\CC[[h,h']]$-algèbre de $\Uhh$ commutant avec $H$. En effet on a
$$ [H, (X^-)^a H^b (X^+)^c] \ = \ 2 \: \! (c-a) \: \! (X^-)^a H^b (X^+)^c \, , \quad \text{ pour } a,b,c \in \NN \, . $$
Puisque $C$ appartient au centre de $\Uhh$ et puisque $[X^+, X^-]$ commute avec $H$, on obtient les deux points de la proposition.
\end{proof}

\begin{proposition} \label{prop_autrepres}
Soient $g \in \NN_{\geq 1}$ et $V_{h,h'}(\mathfrak{sl}_2,g)$ la $\CC[[h,h']]$-algèbre topologiquement engendrée par $X^\pm$, $H$ et définie par les relations
$$ [H, X^{\pm}] \ = \ \pm 2 \, X^{\pm} \, , \quad \quad [X^+, X^-] \ = \ \sum_{n,m \, \in \, \NN} \alpha_{n,m} \, h^n (h')^m \, . $$
Il existe un isomorphisme de $\CC[[h,h']]$-algèbre
$$ \varphi \, : \ V_{h,h'}(\mathfrak{sl}_2,g) \ \stackrel{\sim}{\longrightarrow} \ \Uhh \, , $$
uniquement déterminé par $\varphi(X^{\pm}) = X^{\pm}$ et $\varphi(H) = H$. En outre, $\varphi$ vérifie :
$$ \varphi^{-1}(C) \ = \ \sum_{n,m \in \NN} C_{n,m} \, h^n (h')^m \, .$$
\end{proposition}

\begin{proof}
D'après la définition des $\alpha_{n,m}$, il existe un unique morphisme d'algèbre
$$ \varphi \, : \ V_{h,h'}(\mathfrak{sl}_2,g) \ \to \ \Uhh \, , $$
vérifiant $\varphi(X^{\pm}) = X^{\pm}$ et $\varphi(H) = H$. \\
D'après les relations définissant l'algèbre $V_{h,h'}(\mathfrak{sl}_2,g)$, on voit que cette dernière est topologiquement engendrée par les monômes de la forme
$$ (X^-)^a H^b (X^+)^c \, , \quad \text{ avec } a,b,c \in \NN \, . $$
La proposition \ref{prop_basehh} implique alors que $\varphi$ est un isomorphisme. \\
On a $\varphi^{-1} (X^{\pm}) = X^{\pm}$ et $\varphi^{-1} (H) = H$, donc, d'après la définition des $C_{n,m}$, on voit que $\varphi^{-1}(C) = \sum_{n,m \in \NN} C_{n,m} \, h^n (h')^m$.
\end{proof}

\begin{rem} \label{rem_coeff}
\begin{itemize}
\item[1)] $\Uhh$ est une déformation de $U(\mathfrak{sl}_2)$, où la limite classique de $C$ vérifie la relation \eqref{eq_casimir} : on le vérifie en calculant la limite classique des relations \eqref{eq_defx} de $\Uhh$. On a donc
\begin{equation} \label{eq_casimircoeff}
C \ = \ 4 \: \! X^+ X^- + (H-1)^2 \ + \sum_{n+m \, \geq \, 1} C_{n,m} \, h^n (h')^m \, .
\end{equation}

\item[2)] $\Uhh$ est une déformation de $\Uh$, par suite :
$$ [X^+, X^-] \ = \ \frac{\exp (h \: \! H) - \exp (-h \: \! H)}{\exp (h) - \exp (-h)} \ + \ h' \! \sum_{n,m \, \in \, \NN} \alpha_{n,m+1} \, h^n (h')^m \, . $$
\end{itemize}
\end{rem}

\subsection{Symétries}
On donne ici des propriétés de symétrie de l'algèbre $U_{\: \! h,h'}(\mathfrak{sl}_2,g)$, analogues naturels de propriétés qui existent déjà pour $U(\mathfrak{sl}_2)$ et $\Uh$. Les résultats ici sont nouveaux.

\begin{lemma}
\begin{itemize}
\item[1)] Il existe un unique automorphisme $\omega$ de la $\CC[[h,h']]$-algèbre $U_{\: \! h,h'}(\mathfrak{sl}_2,g)$ tel que $\omega(X^{\pm}) = X^{\mp}$, $\omega(H) = -H$. Il est involutif et envoie $C$ sur $C$.

\item[2)] Il existe un unique anti-automorphisme $\tau$ de la $\CC[[h,h']]$-algèbre $U_{\: \! h,h'}(\mathfrak{sl}_2,g)$ tel que $\tau(X^{\pm}) = X^{\pm}$, $\tau(H) = -H$. Il est involutif et envoie $C$ sur $C$.
\end{itemize}
\end{lemma}

\begin{proof}
Dans les deux points, le fait que $U_{h,h'}(\mathfrak{sl}_2,g)$ est topologiquement engendrée par $X^-$, $X^+$, et $H$ implique l'unicité. Reste donc l'existence.
\begin{itemize}
\item[1)] Il suffit de vérifier que
$$ \big( \omega(X^{+}), \omega(X^-), \omega(H), \omega(C) \big) \ = \ \big( X^-, X^+, -H, C \big) $$
vérifient les relations de la définition \ref{def_uhh} dans l'algèbre $U_{h,h'}(\mathfrak{sl}_2,g)$. Faisons-le pour les relations \eqref{eq_defx}, les autres étant évidentes. $\omega(H_e^{\pm}) \ = \ H_{-e}^{\mp}$, par suite
\begin{eqnarray*}
\omega \left( X^{\pm} \right) \omega \left( X^{\mp} \right) \ = \ X^{\mp} X^{\pm} & = & \xxp{\mp}{} \\
& = & \omega \left( \xxp{\pm}{} \right) .
\end{eqnarray*}

\item[2)] Ici, il faut vérifier que
$$ \big( \tau(X^{+}), \tau(X^-), \tau(H), \tau(C) \big) \ = \ \big( X^+, X^-, -H, C \big) $$
vérifient les relations de la définition \ref{def_uhh} dans l'algèbre $U_{h,h'}^{opp}(\mathfrak{sl}_2,g)$, l'algèbre opposée de $U_{h,h'}(\mathfrak{sl}_2,g)$. Faisons-le pour les relations \eqref{eq_defx}, les autres étant évidentes. \\
$\tau (H_e^{\pm}) \ = \ H_{-e}^{\mp}$, par suite
\begin{eqnarray*}
\tau \left( X^{\pm} \right) \cdot_{opp} \left( X^{\mp} \right) \ = \ X^{\mp} X^{\pm} & = & \xxp{\mp}{} \\
& = & \tau \left( \xxp{\pm}{} \right) .
\end{eqnarray*}
\end{itemize}
\end{proof}

\subsection{Interpolation}
On définit ici rigoureusement ce que signifie spécialiser $\Uhh$ en $Q = \varepsilon$. Une fois les définitions posées, on prouve le lemme \ref{lem_fond}. Ce résultat est crucial dans notre étude, il exprime que toutes les relations de $\Urh{\! gh'}$ peuvent être obtenues à partir de la spécialisation à $Q = \varepsilon$ de $\Uhh$. \\

On note $\mathcal{A}$ la sous-$\CC$-algèbre de $\CC[[h]]$ engendrée par $Q^{\pm 1}$, canoniquement isomorphe aux polynômes de Laurent en $Q$ :
$$ \Aq \ \simeq \ \CC \! \left[ Q^{\pm 1} \right] . $$
On rappelle que $\varepsilon $ désigne la racine $(2g)$-ième primitive de l'unité $\text{e}^{i \pi / g} \in \CC$ et on pose
$$ {[g-1]}^!_{\varepsilon} \ := \ \prod_{k=1}^{g-1} \, {[k]}_{\varepsilon} \ = \ \prod_{k=1}^{g-1} \, \frac{\varepsilon^k - \varepsilon^{-k}}{\varepsilon - \varepsilon^{-1}} \ \in \ \CC^{\ast} \, . $$

\begin{definition}
Soit $g \in \NN_{\geq 1}$. \\
$U_{\mathcal{A},h'} (\mathfrak{sl}_2,g)$ est la sous $\mathcal{A}[[h']]$-algèbre de $U_{h,h'} (\mathfrak{sl}_2,g)$ topologiquement engendrée (pour la topologie ($h'$)-adique) par
$$ C \, , \ H \, , \ X^{\pm} \quad \text{et} \quad Q^{\pm H} \, , \ Q^{\sqrt{C}} + Q^{-\sqrt{C}} \, . $$
\end{definition}

La démonstration du lemme suivant est le passage le plus technique de l'article. Le lemme \ref{lem_fond} assure essentiellement que la propriété d'interpolation de $\Uhh$ en $Q = \varepsilon$ est vérifiée. Il est le noyau dur qui permettra de montrer plus loin dans l'article que la spécialisation à $Q = \varepsilon$ de $\Uhh$ contient $\Urh{gh'}$ comme sous-quotient. Par ailleurs, il généralise à tout $g \in \NN_{\geq 1}$ un résultat analogue dans \cite{hernandez}. Sa preuve n'est en toutefois pas une simple généralisation/adaptation. \\
Dans \cite{hernandez} l'algèbre quantique duale de Langlands est réalisée de plusieurs manières différentes comme sous-quotients de la spécialisation à $Q = \varepsilon$ du groupe quantique d'interpolation de Langlands. Ici, la réalisation est plus universelle puisqu'elle ne fait intervenir qu'un seul sous-quotient. De manière plus précise, on obtient l'algèbre quantique duale en imposant à $\Uhh$ moins de relations qu'il n'est fait dans \cite{hernandez} pour chacune des réalisations.

\begin{lemma} \label{lem_fond}
Dans la $\CC[[h']]$-algèbre
$$ U_{\mathcal{A}, h'} (\mathfrak{sl}_2, g) \ / \ \overline{ \big( Q = \varepsilon, \, Q^{2H} = 1, \, Q^{2 \sqrt{C}} + Q^{-2 \sqrt{C}} = \varepsilon^{2} + \varepsilon^{-2} \big) }^{\, h'} $$
où $\overline{ \ \cdot \ }^{\: \! h'}$ désigne l'adhérence pour la topologie $(h')$-adique, on a l'égalité :
\begin{equation} \label{eqlemfond}
\big( \varepsilon \, T - \varepsilon^{-1} \: \! T^{-1} \big)^2 \, \Big[ (X^+)^g , (X^-)^g \Big] \ = \ ({[g-1]}_{\varepsilon} ^!)^2 \, \big( T^g - T^{-g} \big) \big( T^H - T^{-H} \big) \, .
\end{equation}
\end{lemma}

\begin{proof}
Dans $\Uhh$, la relation \eqref{eq_defx} implique l'égalité
\begin{equation*}
(X^\pm)^g (X^\mp)^g \ = \ \prod_{k=0}^{g-1} \ \prod_{e=1,-1} \frac{\xxn{e}{\pm}{\mp ek}{}}{\xxd{e}{\pm}{\mp ek}{}} \, .
\end{equation*}
Pour $0 \leq k \leq g-1$, on peut considérer les éléments
\begin{multline} \label{eq_defxg}
\check{\Pi}_k \ := \ \prod_{e=1,-1} \ \Bigg[ \xxd{e}{\pm}{\mp ek}{} \Bigg] \\
\text{et } \quad \hat{\Pi}_k \ := \ \prod_{e=1,-1} \ \Bigg[ \xxn{e}{\pm}{\mp ek}{} \Bigg]
\end{multline}
dans l'algèbre $\mathcal{A}[H,\sqrt{C}, Q^{\pm H}, Q^{\pm \sqrt{C}}] [[h']]$. On remarque alors qu'ils invariants par la transformation $Q^{\sqrt{C}} \mapsto Q^{- \sqrt{C}}$. Par conséquent Ils appartiennent en fait à la sous-algèbre $\mathcal{A}[H,\sqrt{C}, Q^{\pm H}, Q^{\sqrt{C}} + Q^{-\sqrt{C}}] [[h']]$. En particulier, leurs images dans $\Uhh$ sont des éléments de $U_{\mathcal{A},h'} (\mathfrak{sl}_2,g)$. \\

La relation \eqref{eq_defxg} implique dans $U_{\mathcal{A},h'} (\mathfrak{sl}_2,g)$ la relation
\begin{multline} \label{eq_defax}
\left( \prod_{k=0}^{g-1} \ \prod_{e = 1, -1} \left[ \xxd{e}{\pm}{\mp ek}{} \right] \right) (X^\pm)^g (X^\mp)^g \\
= \ \prod_{k=0}^{g-1} \ \prod_{e = 1, -1} \left[ \xxn{e}{\pm}{\mp ek}{} \right] \, . \quad \quad
\end{multline}

Pour $0 \leq k \leq g-1$, les éléments $\check{\Pi}_k$ et $\hat{\Pi}_k$ sont invariants par la transformation $Q^H \mapsto -Q^H$, ainsi que par la transformation $Q^{\sqrt{C}} \mapsto -Q^{\sqrt{C}}$. Ils appartiennent donc à la sous-algèbre $\mathcal{A}[H,C,Q^{\pm 2H}, Q^{2\sqrt{C}} + Q^{-2\sqrt{C}}] [[h']]$.

Par ailleurs le morphisme de $\mathcal{A}$-algèbre
$$ \mathcal{A}[Q^{\pm 2H}, Q^{2\sqrt{C}} + Q^{-2\sqrt{C}}] \ \longrightarrow \ \mathcal{A}[Q^{\pm H_e^a}]_{e \in \{ -1, 1 \}, \, a \in \{-, + \} } $$
défini par
$$ Q^{\pm 2H} \ \mapsto \ Q^{2 H_1^+} Q^{-2H_{-1}^-} \quad \text{ et } \quad Q^{2 \sqrt{C}} + Q^{-2 \sqrt{C}} \ \mapsto \ Q^{2H_1^+} Q^{2 H_{-1}^+} + Q^{-2H_1^+} Q^{-2 H_{-1}^+} $$
induit un isomorphisme
\begin{multline*} \mathcal{A} \ \simeq \ \mathcal{A}[Q^{\pm 2H}, Q^{2\sqrt{C}} + Q^{-2\sqrt{C}}] \ / \ \big(Q^{2H} = 1, Q^{2\sqrt{C}} + Q^{-2\sqrt{C}} = \varepsilon^2 + \varepsilon^{-2} \big) \\
\stackrel{\sim}{\longrightarrow} \quad \mathcal{A}[Q^{\pm H_e^a}]_{e \in \{ -1, 1 \}, \, a \in \{-, + \} } \ / \ \big( Q^{H_e^a} = \varepsilon^{(1 - a \: \! e) / 2} \big) \ \simeq \ \mathcal{A} \, ,
\end{multline*}
et par suite un isomorphisme de $\CC[H,C][[h']]$-algèbre
\begin{equation*}
\mathcal{M} \ := \ \left( \frac{\mathcal{A} \Big[ H,C,Q^{\pm 2H}, Q^{2\sqrt{C}} + Q^{-2\sqrt{C}} \Big]}{\Big( Q=\varepsilon, \, Q^{2H} = 1, \, Q^{2\sqrt{C}} + Q^{-2\sqrt{C}} = \varepsilon^2 + \varepsilon^{-2} \Big)} \right) [[h']] \ \stackrel{\sim}{\longrightarrow} \ \left( \frac{\mathcal{A} \Big[ H,C,Q^{\pm H_e^a} \Big]_{e \in \{ -1, 1 \}, \, a \in \{-, + \} }}{\Big( Q = \varepsilon, \, Q^{H_e^a} = \varepsilon^{(1-a \: \! e)/2} \Big)} \right) [[h']] \, .
\end{equation*}

Par conséquent, dans $\mathcal{M}$, on a
\begin{equation} \label{eq_final1}
\prod_{k=0}^{g-1} \ \prod_{e=1,-1} \ \Bigg[ \xxd{e}{\pm}{\mp ek}{} \Bigg] \ = \ (\varepsilon - \varepsilon^{-1})^{2g - 2} \, \big( \varepsilon \, T - \varepsilon^{-1} \: \! T^{-1} \big)^2
\end{equation}

\begin{multline} \label{eq_final2}
\text{et} \quad \prod_{k=0}^{g-1} \ \prod_{e=1,-1} \Bigg[ \xxn{e}{\pm}{\mp ek}{} \Bigg] \\
= \ \Bigg[ \prod_{k=1}^{g-1} \left( \varepsilon^k - \varepsilon^{-k} \right)^2 \Bigg] \prod_{e=1,-1} \Big( T^{\: \! H_e^{\pm} + (g -1 \mp eg \pm e)/2} - T^{\: \! - H_e^{\pm} - (g -1 \mp eg \pm e)/2} \Big) \\
= \ \Bigg[ \prod_{k=1}^{g-1} \left( \varepsilon^k - \varepsilon^{-k} \right)^2 \Bigg]  \Big( T^{\: \! \sqrt{C} +g -1} \ + \ T^{\: \! - \sqrt{C} -g +1} \ - \ T^{\: \!  \mp g} \, T^{\: \! H} \ - \ T^{\: \! \pm g} \, T^{\: \! -H} \Big) \, .
\end{multline}

Dans le quotient $U_{\! A, h'} (\mathfrak{sl}_2, g) \ / \ \overline{ \big( Q = \varepsilon, \, Q^{2H} = 1, \, Q^{2 \sqrt{C}} + Q^{-2 \sqrt{C}} = \varepsilon^{2} + \varepsilon^{-2} \big) }^{\, h'}$, les relations \eqref{eq_final1}  et \eqref{eq_final2} impliquent l'égalité
\begin{multline} \tag{$\ast_{\pm}$}
(\varepsilon - \varepsilon^{-1})^{2g - 2} \, \big( \varepsilon \, T - \varepsilon^{-1} \: \! T^{-1} \big)^2 \, (X^{\pm})^g (X^{\mp})^g \\
= \ \Bigg[ \prod_{k=1}^{g-1} \left( \varepsilon^k - \varepsilon^{-k} \right)^2 \Bigg]  \Big( T^{\: \! \sqrt{C} +g -1} \ + \ T^{\: \! - \sqrt{C} -g +1} \ - \ T^{\: \!  \mp g} \, T^{\: \! H} \ - \ T^{\: \! \pm g} \, T^{\: \! -H} \Big) \, . \quad \quad \quad
\end{multline}
On obtient \eqref{eqlemfond} en faisant la différence de $(\ast_+)$ et $(\ast_-)$.
\end{proof}

\begin{rem} \label{rem_fond}
D'après le lemme précédent, dans le localisé par rapport à $h'$
$$ \Bigg( \, U_{\mathcal{A}, h'} (\mathfrak{sl}_2, g) \ / \ \overline{ \big( Q = \varepsilon, \, Q^{2H} = 1, \, Q^{2 \sqrt{C}} + Q^{-2 \sqrt{C}} = \varepsilon^{2} + \varepsilon^{-2} \big) }^{\, h'} \, \Bigg)_{\! \! h'} \, , $$
les éléments
$$ {}^L \! X^+ \ := \ ({[g-1]}_{\varepsilon}^!)^{-2} \, \big( \varepsilon \, T -  \varepsilon^{-1} T^{-1} \big)^2 \,  \big( T^g - T^{-g} \big)^{-2} \, \, (X^+)^g \, , $$
$$  {}^L \! X^- \ := (X^-)^g \, , \quad \ \text{ et } \quad {}^L \! H \ := \ \frac{H}{g} $$
vérifient
$$ [ {}^L \! H, {}^L \! X^{\pm}] \ = \ \pm 2 \,  {}^L \! X^{\pm} \, , \quad  [  {}^L \! X^+ ,  {}^L \! X^- ] \ = \ \frac{\text{sinh} (g h' \: \!  {}^L \! H)}{\text{sinh} (g h')} \, . $$
Autrement dit, la sous-$\CC[[h']]$-algèbre engendrée par ${}^L \! X^{\pm}$ et ${}^L \! H$, que nous noterons $\langle {}^L \! X^{\pm}, {}^L \! H \rangle$, est un quotient de $U_{gh'}(\mathfrak{sl}_2)$. \\
Nous verrons, après avoir établi la dualité de Langlands pour les représentations des groupes quantiques de rang 1, que $\langle {}^L \! X^{\pm}, {}^L \! H \rangle$ est en fait isomorphe à $U_{gh'}(\mathfrak{sl}_2)$ : voir le théorème \ref{thm_uhhinterpolation}.
\end{rem}

\section{Théorie des Représentations des Groupes Quantiques d'Interpolation} \label{section_repr}
Dans cette section, on étudie la théorie des représentations de $\Uhh$. On commence, en utilisant la décomposition triangulaire de $\Uhh$ (proposition \ref{prop_triang}), par construire des modules de Verma de $\Uhh$, déformations selon $h'$ pour chaque $g \in \NN_{\geq 1}$, des modules de Verma de $\Uh$ correspondants. Comme pour $\mathfrak{sl}_2$ et $\Uh$, on peut en donner une description explicite. Voir la proposition \ref{prop_vermadescr}. \\
Pour chaque $g \in \NN_{\geq 1}$, on définit ensuite une déformation selon $h'$ des représentations de $\Uh$ de rang fini $L^h(n)$ ($n \in \NN$). Voir la proposition \ref{prop_indecdescr}. Comme pour $\mathfrak{sl}_2$ et $\Uh$, cette déformation peut-être réalisée comme quotient du module de Verma de $\Uhh$. \\
On utilise ensuite les outils de la section \ref{section_deformations} pour décrire précisément la catégorie des représentations de rang fini de $\Uhh$ : voir le théorème \ref{thm_rephh}. On obtient notamment que toute représentation de $\mathfrak{sl}_2$ de dimension finie peut-être déformée en une représentation de $\Uhh$, que la catégorie possède la propriété de Krull-Schmidt (toute représentation est, d'une manière unique, somme directe de représentations indécomposables), et que ses seules représentations indécomposables sont les déformations de $L^h(n)$ ($n \in \NN$) mentionnées précédemment. \\
Ces résultats sur la catégorie des représentations de rang fini des groupes quantiques d'interpolation de Langlands de rang 1 sont nouveaux. Dans \cite{hernandez}, pour $g=2,3$ des modules de Verma sont définis, ainsi que des déformations des représentations irréductibles de dimension finie $L^q(n)$ de $U_q(\mathfrak{sl}_2)$. Toutefois une étude systématique de la théorie des représentations de dimension finie des groupes quantiques d'interpolation de rang 1 n'est pas faite. Celle-ci semble difficilement réalisable du point de vue de \cite{hernandez}, du fait notamment du nombre trop important de générateurs des groupes quantiques d'interpolation. \\

Soit $g \in \NN_{\geq 1}$. Le groupe quantique d'interpolation $\Uhh$ est, d'après le théorème \ref{thm_uhhdefor}, une déformation de $U(\mathfrak{sl}_2)$. En reprenant les définitions de la section \ref{section_deformations}, on rappelle qu'une représentation $V^{h,h'}$ de $\Uhh$ est donnée par un $\CC$-espace vectoriel $V$ et par une action de $\Uhh$ sur $V[[h,h']]$. \\

La représentation $V^{h,h'}$ est une représentation de poids si le $\CC[[h,h']]$-module $V^{h,h'}$ admet une décomposition en somme directe :
\begin{equation}
V^{h,h'} \ = \ \bigoplus_{n \in \ZZ} \, V^{h,h'}_n \, , \quad \text{ avec } \ V^{h,h'}_n \ := \ \{ f \in V^{h,h'} : \, H.f = n \: \! f \} \, .
\end{equation}
Remarquons que $V^{h,h'}_n$ ($n \in \ZZ$) est un sous-$\CC[[h,h']]$-module fermé de $V^{h,h'}$ (en effet, $H$ est $\CC[[h,h']]$-linéaire, donc continu). \\

Un élément $f \in V^{h,h'}_n$ est appelé un élément de poids $n$, et si $X^+.f = 0$, $f$ est appelé un élément de plus haut poids $n$. \\
Si $V^{h,h'}_n \neq (0)$, alors $n$ est appelé un poids de $V^{h,h'}$ et $V^{h,h'}_n$ un espace de poids de $V^{h,h'}$. On notera $\text{wt} \: \! V^{h,h'}$ l'ensemble des poids de $V^{h,h'}$. \\

La catégorie additive $\CC[[h,h']]$-linéaire des représentations de $\Uhh$ de rang fini sera notée $\cmodhh$. \\

On pose $U := \Uc$. On rappelle que le module de Verma $M(n)$ de $\Uc$ ($n \in \ZZ$) est donné par :
$$ M(n) \ := \ U \ / \, \Big( U \: \! \bar{X}^+ + \ U \: \! (\bar{H} - n \: \! 1) \Big) \, . $$
En notant $m_0 \in M(n)$ l'image de $1 \in U$, et $m_j := (\bar{X}^-)^j. m_0$ ($j \in \NN)$, on obtient une base $(m_j)_{j \in \NN}$ de $M(n)$ telle que
\begin{eqnarray*}
H.m_j & = & (n - 2j) \, m_j \, , \\
X^- . m_j & = & m_{j+1} \, , \\
X^+ . m_j & = & j \: \! (n -j+1) \, m_{j-1} \, , \quad \quad \text{(on pose $m_{-1} := 0$).} \\
\end{eqnarray*}

On pose $U^h := \Uh$. Le module de Verma $M^h(n)$ de $\Uh$ ($n \in \ZZ$) est donné par :
$$ M^h(n) \ := \ U^h \ / \, \Big( U^h \: \! \tilde{X}^+ + \ U^h \: \! (\tilde{H} - n \: \! 1) \Big) \, . $$

On pose $U^{h,h'} := \Uhh$ et on définit le module de Verma $M^{h,h'} (n,g)$ de $\Uhh$ ($n \in \ZZ$) :
\begin{equation} \label{eq_verma}
M^{h,h'}(n,g) \ := \ U^{h,h'} \ / \, \Big( U^{h,h'} \: \! X^+ + \ U^{h,h'} \: \! (H - n \: \! 1) \Big) \, .
\end{equation}

$\Uhh = U(\mathfrak{sl_2})[[h,h']]$ (on rappelle que cette égalité doit être interprétée en considérant la base donnée dans la proposition \ref{prop_basehh}) est une représentation, appelée régulière, de $\Uhh$ en considérant l'action par multiplication à gauche. C'est une déformation de la représentation régulière de $U(\mathfrak{sl}_2)$. \\

Puisque $X^+ H = H X^+ - 2 \: \! X^+$, on a l'égalité de $\CC[[h,h']]$-module
$$ U^{h,h'} \: \! X^+ + \ U^{h,h'} \: \! (H - n \: \! 1) \ = \ \Big( U \: \! \bar{X}^+ + \ U \: \! (\bar{H} - n \: \! 1) \Big)[[h,h']] \, . $$
Le sous-$\Uhh$-module $U^{h,h'} \: \! X^+ + \ U^{h,h'} \: \! (H - n \: \! 1)$ de la représentation régulière est donc une sous-représentation, et par suite une déformation formelle de la représentation $U \: \! \bar{X}^+ + \ U \: \! (\bar{H} - n \: \! 1)$ de $\Uc$. \\
On en déduit que $M^{h,h'}(n,g) = M(n)[[h,h']]$ est une déformation formelle du module de Verma $M(n)$ de $\Uc$ selon $h$ et $h'$ (remarquons que cela prouve en particulier que $M^{h,h'}(n,g)$ est non trivial). \\

De la même façon, $M^h(n) = M(n)[[h]]$ ($n \in \ZZ$) est une déformation du module de Verma $M(n)$ de $\Uc$ selon $h$ et $M^{h,h'}(n,g)$ est une déformation de $M^h(n)$ selon $h'$. \\

On résume ce qui précède, et on donne une description explicite du module $M^{h,h'}(n,g)$, dans la proposition \ref{prop_vermadescr} suivante. 

\begin{proposition} \label{prop_vermadescr}
Soient $n \in \ZZ$ et $g \in \NN_{\geq 1}$.
\begin{itemize}
\item[1)] $M^{h,h'}(n,g)$ est une représentation de rang infini de $\Uhh$.
\item[2)] $M^{h,h'}(n,g)$ est une déformation formelle du module de Verma $M(n)$ de $\Uc$ selon $h$ et $h'$.
\item[3)] $M^{h,h'}(n,g)$ est une déformation formelle du module de Verma $M^h(n)$ de $\Uh$ selon $h'$.
\item[4)] Le module de Verma $M^{h,h'}(n,g)$ est une représentation de poids.
\item[5)] L'action de $\Uhh$ sur $M^{h,h'}(n,g)$ est donnée dans la base topologique $(m_j)_{j \in \NN}$ par :
\begin{eqnarray*}
H.m_j & = & (n - 2j) \, m_j \, , \\
C. m_j & = & (n + 1)^2 \, m_j \, , \\
X^- . m_j & = & m_{j+1} \, , \\
X^+ . m_j & = & \big[ j \big]_{\! Q \: \! T^{ {\{ j \}}}} \, \big[ n -j+1 \big]_{\! Q \: \! T^{ {\{ n -j+1 \}}}} \, m_{j-1} \, ,
\end{eqnarray*}
\end{itemize}
\end{proposition}

\begin{proof}
Les trois premiers points ont déjà été prouvés. \\
On utilise les relations de la définition \ref{def_uhh} de $\Uhh$ et le fait que $m_j = (X^-)^j . m_0$ ($j \in \NN$) pour montrer le dernier point. \\
En examinant l'action de $H$, on obtient que $M^{h,h'}(n,g)$ est une représentation de poids avec $M^{h,h'}_{n-2j}(n,g) = \CC[[h,h']]. m_j$ ($j \in \NN$) et
$$ M^{h,h'}(n,g) \ = \ \bigoplus_{j \in \NN} \, M^{h,h'}_{n - 2j}(n,g) \, . $$
\end{proof}

L'action de $\Uh$ sur $M^h(n)$ ($n \in \ZZ$) est donnée dans la base topologique $(m_j)_{j \in \NN}$ par :
\begin{eqnarray*}
H.m_j & = & (n - 2j) \, m_j \, , \\
X^- . m_j & = & m_{j+1} \, , \\
X^+ . m_j & = & {[j]}_Q \, {[n-j+1]}_Q \, m_{j-1} \, ,
\end{eqnarray*}

Formellement donc, le module de Verma $M^{h,h'}(n,g)$ de $\Uhh$ s'obtient à partir du module de Verma $M^h(n)$ de $U_h(\mathfrak{sl}_2)$, en remplaçant les boîtes quantiques $[a]_Q$ par les boîtes ${[a]}_{Q \: \! T^{\{a \}}}$ ($a \in \ZZ)$.  \\

Les modules de Verma de $\Uhh$ vérifient la propriété universelle suivante. Sa démonstration s'obtient immédiatement à partir de \eqref{eq_verma}.

\begin{proposition} \label{prop_verma}
Si $V^{h,h'}$ est une représentation de $\Uhh$ et $f \in V^{h,h'}$ un élément de plus haut poids $n$, alors il existe un unique morphisme
$$ \varphi \, : \ M^{h,h'}(n,g) \ \to \ V^{h,h'} $$
qui vérifie $\varphi \! \left( m_0 \right) = f$.
\end{proposition}

On rappelle que la représentation irréductible de dimension finie $L(n)$ de $\Uc$ est obtenue comme quotient de $M(n)$ par la sous-représentation $\bigoplus_{j \geq n+1} \CC \, m_j$. \\
$L(n)$ admet alors comme base la famille des vecteurs $m_j$ ($0 \leq j \leq n$) et l'action de $\Uc$ sur $L(n)$ est donnée par
\begin{eqnarray*}
H.m_j & = & (n - 2j) \, m_j \, , \\
X^- . m_j & = & m_{j+1} \, , \\
X^+ . m_j & = & j \: \! (n -j+1) \, m_{j-1} \, ,
\end{eqnarray*}
(on pose $m_{-1} = m_{n+1} := 0$). \\

Une sous-représentation $V[[h,h']]$ de $M^{h,h'}(n,g) = M(n)[[h,h']]$ vérifie en particulier $H.V \subset V$. L'action de $H$ sur $M(n)$ étant localement finie et semi-simple, elle l'est aussi sur $V$. Par suite $V$ admet comme base une sous-famille de $(m_j)_{j \in \NN}$. \\

D'après le point 5) de la proposition \ref{prop_vermadescr}, $M^{h,h'}(n,g)$ admet une sous-représentation non triviale si et seulement si $n \in \NN$. Dans ce dernier cas, il existe une unique sous-représentation non triviale : $\left( \bigoplus_{j \geq n+1} \CC \, m_j \right) [[h,h']]$ qui est isomorphe à $M^{h,h'}(-n-2)$ d'après la proposition \ref{prop_verma}. \\
On note $L^{h,h'}(n,g)$ la représentation quotient, de rang fini égal à $n+1$. \\
On définit de la même manière la représentation de rang fini $L^h(n)$ de $\Uh$. \\

Par des arguments similaires à ceux donnés dans le cas des modules de Verma, on vérifie que $L^{h,h'}(n,g)$ est une déformation selon $h$ et $h'$ de la représentation irréductible $L(n)$ de $U(\mathfrak{sl}_2)$, et une déformation selon $h'$ de la représentation $L^h(n)$ de $\Uh$. \\

On résume ce qui précède, et on donne une description explicite de $L^{h,h'}(n,g)$, dans la proposition suivante.

\begin{proposition} \label{prop_indecdescr}
Soient $n \in \NN$ et $g \in \NN_{\geq 1}$.
\begin{itemize}
\item[1)] $L^{h,h'}(n,g)$ est une représentation de $\Uhh$ de rang fini égal à $n+1$.
\item[2)] $L^{h,h'}(n,g)$ est une déformation formelle de la représentation irréductible $L(n)$ de $\Uc$ selon $h$ et $h'$.
\item[3)] $L^{h,h'}(n,g)$ est une déformation formelle de la représentation $L^h(n)$ de $\Uh$ selon $h'$.
\item[4)] $L^{h,h'}(n,g)$ est une représentation de poids.
\item[5)] L'action de $\Uhh$ sur $L^{h,h'}(n,g)$ est donnée dans la base topologique $(m_j)_{0 \leq j \leq n}$ par :
\begin{eqnarray*}
H.m_j & = & (n - 2j) \, m_j \, , \\
C. m_j & = & (n + 1)^2 \, m_j \, , \\
X^- . m_j & = & m_{j+1} \, , \\
X^+ . m_j & = & \big[ j \big]_{\! Q \: \! T^{ {\{ j \}}}} \, \big[ n -j+1 \big]_{\! Q \: \! T^{ {\{ n -j+1 \}}}} \, m_{j-1} \, ,
\end{eqnarray*}
\end{itemize}
\end{proposition}

\begin{proof}
Les trois premiers points ont déjà été prouvés. \\
Le dernier point est une conséquence du point 5) de la proposition \ref{prop_vermadescr}. \\
En examinant l'action de $H$, on obtient que $L^{h,h'}(n,g)$ est une représentation de poids avec $L^{h,h'}_{n-2j}(n,g) = \CC[[h,h']]. m_j$ ($0 \leq j \leq n$) et
$$ L^{h,h'}(n,g) \ = \ \bigoplus_{0 \leq j \leq n} \, L^{h,h'}_{n - 2j}(n,g) \, . $$
\end{proof}

L'action de $\Uh$ sur la représentation $L^h(n)$ ($n \in \NN$) est donnée par :
\begin{eqnarray*}
H.m_j & = & (n-2j) \, m_j \, , \\
X^- . m_j & = & m_{j+1} \, , \\
X^+ . m_j & = & [j \big]_Q \: \!  [n -j+1]_Q \, m_{j-1} \, .
\end{eqnarray*}

Ici encore, il suffit de remplacer, comme expliqué plus haut, les boîtes quantiques de la représentation $L^h(n)$ de $U_h(\mathfrak{sl}_2)$ pour obtenir la représentation $L^{h,h'}(n,g)$ de $U_{h,h'}(\mathfrak{sl}_2)$.

\begin{rem}
Pour $g= 1,2,3$, des analogues des représentations $M^{h,h'}(n,g)$ et $L^{h,h'}(n,g)$ existent dans \cite{hernandez}. Toutefois celui de $M^{h,h'}(n,g)$ n'apparaît pas comme un module de Verma (les groupes quantiques d'interpolation de Langlands dans \cite{hernandez} n'admettent a priori pas de décomposition triangulaire), quotient de la représentation régulière.
\end{rem}

On note $\cmod$ la catégorie abélienne des représentations de dimension finie de l'algèbre de Lie $\mathfrak{sl}_2$. On rappelle que $\cmod$ est une catégorie semi-simple : voir par exemple \cite[II.6]{humphreys}. Semi-simple ici signifie que toute représentation de dimension finie de $\mathfrak{sl}_2$ est une somme directe finie de représentations irréductibles. \\

D'après la section \ref{section_deformations} sur les déformations formelles, on dispose du foncteur additif : 
$$ \lim_{h,h' \to 0} : \ \cmodhh \, \longrightarrow \ \cmod \, . $$

On note $\grot$ et $\grothh$ les groupes de Grothendieck des catégories additives respectives $\cmod$ et $\cmodhh$. \\
On dira qu'une catégorie additive $\mathcal C$ possède la propriété de Krull-Schmidt si chaque objet dans $\mathcal C$ est une somme directe d'objets indécomposables, les facteurs composant la somme étant uniques à isomorphisme et ordre près. \\

Le théorème suivant donne plusieurs résultats nouveaux sur la catégorie des représentations de rang fini de $\Uhh$. Ceux-ci n'ont pas d'analogues dans \cite{hernandez}.

\begin{theorem} \label{thm_rephh}
\begin{itemize}
\item[1)] Toute représentation $V \in \cmodh$ admet une déformation $V^{h,h'}$ dans $\cmodhh$.
\item[2)] Deux représentations dans $\cmodhh$ sont isomorphes si et seulement si leurs limites classiques le sont.
\item[3)] Le foncteur $\lim_{h,h' \to 0} : \cmodhh \to \cmod$ induit un isomorphisme additif
$$ \left[ \lim_{h,h' \to 0} \right] : \ \grothh \, \stackrel{\sim}{\longrightarrow} \ \grot $$
du groupe de Grothendieck $\grothh$ sur le groupe de Grothendieck $\grot$.
\item[4)] La catégorie $\cmodhh$ possède la propriété de Krull-Schmidt.
\item[5)] Pour tout $n \in \NN$, la représentation de rang fini $L^{h,h'}(n,g)$ est indécomposable. Pour toute représentation $V^{h,h'}$ de $\Uhh$ de rang fini indécomposable, il existe un unique $n \in \NN$ tel que $V^{h,h'}$ est isomorphe à $L^{h,h'}(n,g)$.
\item[6)] Toute représentation dans $\cmodhh$ est une représentation de poids.
\end{itemize}
\end{theorem}

\begin{proof}
On note $\mathcal C$ la catégorie des représentations de rang fini de la déformation formelle constante $\Uc[[h,h']]$. On sait d'après le théorème \ref{thm_uhhdefor} que $\Uhh$ est une déformation formelle triviale de $U(\mathfrak{sl}_2)$. Il existe alors une équivalence de catégorie additive $\CC[[h,h']]$-linéaire entre $\cmodhh$ et $\mathcal C$, de telle sorte que le diagramme suivant commute :
\begin{equation} \label{eq_commtriv} \xymatrix{
\mathcal C \ar[rr]^-{\simeq} \ar[dr]_-{\lim_{h,h' \to 0}} && \cmodhh \ar[dl]^-{\quad \lim_{h,h' \to 0}} \\
& \cmod &
} \end{equation}
Soit $V \in \cmod$, d'après le théorème \ref{thm_hochrep} et puisque $H^1 \! \left( U(\mathfrak{sl_2}),\text{End}_{\CC} V \right) = 0$ (voir par exemple \cite[II.11]{guichardet2}), toute déformation $V^{h,h'}$ dans $\mathcal C$ de $V$ est équivalente à la déformation formelle constante de $V$. \\

Après ces remarques, démontrons dans l'ordre les différents points du théorème.
\begin{itemize}
\item[1)] Toute représentation $V \in \cmod$ admet une déformation dans $\mathcal C$ : la déformation constante. On conclut grâce à \eqref{eq_commtriv}. 
\item[2)] D'après \eqref{eq_commtriv} il suffit de montrer le point 2) pour $\mathcal C$. \\
Le foncteur $\lim_{h,h' \to 0}$ induit une application, que l'on note $\left[ \lim_{h,h' \to 0} \right]$, des classes d'isomorphisme de $\mathcal C$ vers les classes d'isomorphisme de $\cmod$. Le foncteur
$$ \mathcal{Q}^{h,h'} \, : \ \cmod \ \to \ \mathcal C \, , $$
qui à une représentation $V \in \cmod$ associe sa déformation constante $V[[h,h']]$, induit clairement un inverse à gauche $\left[ \mathcal{Q}^{h,h'} \right]$ de $\left[ \lim_{h,h' \to 0} \right]$. $\left[ \mathcal{Q}^{h,h'} \right]$ est aussi un inverse à droite, puisque toute représentation dans $\mathcal C$ est isomorphe à la déformation constante de sa limite classique.
\item[3)] Les deux foncteurs précédents étant additifs, $\left[ \lim_{h,h' \to 0} \right]$ et $\left[ \mathcal{Q}^{h,h'} \right]$ induisent, sur les groupes de Grothendieck de $\mathcal C$ et $\cmod$, des isomorphismes inverses l'un de l'autre. On conclut grâce à \eqref{eq_commtriv} encore une fois.
\item[4)] $\cmod$ étant semi-simple, elle possède en particulier la propriété de Krull-Schmidt. D'après le point 2) et puisque $\lim_{h,h' \to 0}$ est additif, on en déduit que $\cmodhh$ la possède aussi.
\item[5)] D'après la proposition \ref{prop_indecdescr}, la limite classique de la représentation $L^{h,h'}(n,g)$ est la représentation irréductible (donc indécomposable) $L(n)$. Toute représentation indécomposable (donc irréductible) de $\mathfrak{sl}_2$ est isomorphe à un $L(n)$ ($n \in \NN$). On conclut en remarquant que le point 2) implique qu'une représentation de $\cmodh$ est indécomposable si et seulement si sa limite classique l'est.
\item[6)] D'après ce qui précède, toute représentation de $\Uhh$ est isomorphe à une somme directe de représentations $L^{h,h'}(n,g)$. On conclut grâce à la proposition \ref{prop_indecdescr}.
\end{itemize}
\end{proof}

On note, en identifiant les $\CC[[h,h']]$-algèbres $U^0_{h,h'}(\mathfrak{sl}_2,g)$ et $\CC[H][[h,h']]$,
$$ \Pi^0 \, : \ \Uhh \ \to \ \CC[H][[h,h']] $$
la projection $\CC[[h,h']]$-linéaire sur $U^0_{h,h'}(\mathfrak{sl}_2,g)$ définie naturellement selon la décomposition triangulaire \eqref{eq_triang}. \\
En composant $\Pi^0$ avec l'unique automorphisme de la $\CC[[h,h']]$-algèbre $\CC[H][[h,h']$ qui à $H$ associe $H+1$, on obtient un nouveau morphisme $\CC[[h,h']]$-linéaire $\delta : \Uhh \to \CC[H][[h,h']]$. \\
En utilisant la remarque \ref{rem_coeff} et la description explicite de l'action de $C$ sur  les représentations $L^{h,h'}(n,g)$ donnée dans la proposition \ref{prop_indecdescr}, on voit que $\delta(C) = H^2$. Le corollaire \ref{cor_eng} implique alors la proposition suivante, qui donne un équivalent pour $\Uhh$ de l'isomorphisme d'Harish-Chandra (voir \cite{harish}, \cite{tanisaki} et \cite[VI.4]{kassel}).

\begin{proposition}
La restriction de $\delta$ sur le centre de $\Uhh$ définit un isomorphisme de $\CC[[h,h']]$-algèbre
$$ \delta \, : \ Z \Big( \Uhh \Big) \ \stackrel{\sim}{\longrightarrow} \, \ \CC[H^2][[h,h']] \, , $$
qui à $C$ associe $H^2$.
\end{proposition}

\section{Dualité de Langlands pour les Représentations des Qroupes Quantiques en Rang 1} \label{section_langlands}
Dans cette section, on étudie la spécialisation à $Q = \varepsilon$ de la représentation indécomposable $L^{h,h'}(n,g)$ et du module de Verma $M^{h,h'}(n,g)$. \\
Il est montré que ces spécialisations contiennent de manière naturelle des représentations du groupe quantique $\Urh{gh'}$ : les représentations Langlands $g$-duales de $L^{h,h'}(n,g)$ et $M^{h,h'}(n,g)$. Voir les théorèmes \ref{thm_reprinter} et \ref{thm_reprinter2}. \\
En particulier, la conjecture \cite[conjecture 1]{hernandez} concernant l'existence de représentations qui déforment simultanément deux représentations Langlands duales, est résolue ici dans le cas du rang 1 et pour tout $g \in \NN_{\geq 1}$. \\
La propriété d'interpolation/dualité pour les représentations permet de terminer la démonstration, entamée par le lemme \ref{lem_fond}, de la propriété d'interpolation de $\Uhh$ : voir le théorème \ref{thm_uhhinterpolation}. \\

Soit $g \in \NN_{\geq 1}$. Rappelons que $\varepsilon$ désigne la racine $(2g)$-ième primitive de l'unité $\text{e}^{i \pi/ g} \in \CC$. 

\begin{definition}
La spécialisation à $Q = \varepsilon$ de $\Uhh$ est la $\CC[[h']]$-algèbre
$$ U_{\varepsilon,h'}(\mathfrak{sl_2},g) \ := \ U_{\Aq,h'}(\mathfrak{sl}_2,g) \, / \, \overline{(Q = \varepsilon)}^{\: \! h'} \, . $$
\end{definition}

On définit aussi une notion de poids pour les $U_{\varepsilon,h'}(\mathfrak{sl_2},g)$-modules. \\
Un module $U_{\varepsilon,h'}(\mathfrak{sl_2},g)$-module $V^{\varepsilon,h'}$ est un module de poids si il existe une décomposition en espace de poids :
\begin{equation}
V^{\varepsilon,h'} \ = \ \bigoplus_{n \in \ZZ} \, V^{\varepsilon,h'}_n \, , \quad \text{ avec } \ V^{\varepsilon,h'}_n \ := \ \{ f \in V^{\varepsilon,h'} : \, H.f = n \: \! f \} \, .
\end{equation}
Un vecteur $f \in V^{\varepsilon,h'}_n$ est appelé un vecteur de poids $n$. Si $V^{\varepsilon,h'}_n \neq (0)$, alors $n$ est appelé un poids de $V^{\varepsilon,h'}$ et $V^{\varepsilon,h'}_n$ un espace de poids de $V$. \\

On note $M^{\Aq,h'}(n,g)$ ($n \in \ZZ$) et $L^{\Aq,h'}(n,g)$ ($n \in \NN$) les sous-$\Aq[[h']]$-modules de respectivement $M^{h,h'}(n,g)$ et $L^{h,h'}(n,g)$, topologiquement engendrés par les bases $(m_j)_{j \in \NN}$ et $(m_j)_{0 \leq j \leq n}$ (pour la topologie $(h')$-adique). On a
$$ M^{\Aq,h'}(n,g) \ \simeq \ \Big( \bigoplus_{j \in \NN} \Aq \, m_j \Big) [[h']] \, , \quad \ L^{\Aq,h'}(n,g) \ \simeq \ \Big( \bigoplus_{0 \leq j \leq n} \Aq \, m_j \Big) [[h']] \, . $$
D'après les propositions \ref{prop_vermadescr} et \ref{prop_indecdescr}, $X^{\pm}$, $H$, $C$, $Q^{\pm H}$ et $Q^{\sqrt{C}} + Q^{-\sqrt{C}}$ stabilisent $M^{\Aq,h'}(n,g)$ et $L^{\Aq,h'}(n,g)$. Par suite les actions de $U_{h,h'} (\mathfrak{sl}_2,g)$ sur $M^{h,h'}(n,g)$ et $L^{h,h'}(n,g)$ induisent des actions de $U_{\Aq,h'} (\mathfrak{sl}_2,g)$ sur $M^{\Aq,h'}(n,g)$ et $L^{\Aq,h'}(n,g)$. \\

On note $M^{\varepsilon,h'}(n,g)$ et $L^{\varepsilon,h'}(n,g)$ les spécialisations à $Q = \varepsilon$ de $M^{\Aq,h'}(n,g)$ et $L^{\Aq,h'}(n,g)$ :
$$ M^{\varepsilon,h'}(n,g) \ := \ M^{\Aq,h'}(n,g) \, / \, \overline{(Q -\varepsilon). M^{\Aq,h'}(n,g)}^{\: \! h'} \, , $$
$$ L^{\varepsilon,h'}(n,g) \ := \ L^{\Aq,h'}(n,g) \, / \, \overline{(Q -\varepsilon). L^{\Aq,h'}(n,g)}^{\: \! h'} \, , $$
où $\overline{\ \cdot \ }^{\: \! h'}$ désignent l'adhérence pour la topologie $(h')$-adique. $M^{\varepsilon,h'}(n,g)$ et $L^{\varepsilon,h'}(n,g)$ sont en particulier des $\CC[[h']]$-modules et on a
\begin{equation} \label{eq_vermaindecspe}
M^{\varepsilon,h'}(n,g) \ \simeq \ \Big( \bigoplus_{j \in \NN} \CC \, m_j \Big) [[h']] \, , \quad \ L^{\varepsilon,h'}(n,g) \ \simeq \ \Big( \bigoplus_{0 \leq j \leq n} \CC \, m_j \Big) [[h']] \, .
\end{equation}
L'action de $U_{\Aq,h'} (\mathfrak{sl}_2,g)$ sur $M^{\Aq,h'}(n,g)$ induit une action de $U_{\varepsilon,h'}(\mathfrak{sl_2},g)$ sur $M^{\varepsilon,h'}(n,g)$. De même l'action de $U_{\Aq,h'} (\mathfrak{sl}_2,g)$ sur $L^{\Aq,h'}(n,g)$ induit une action de $U_{\varepsilon,h'}(\mathfrak{sl_2},g)$ sur $L^{\varepsilon,h'}(n,g)$.

\begin{definition}
Soit $g \in \NN_{\geq 1}$.
\begin{itemize}
\item[1)] Soit $n \in \ZZ$. Le $U_{\varepsilon,h'}(\mathfrak{sl_2},g)$-module $M^{\varepsilon,h'}(n,g)$ est appelé la spécialisation à $Q = \varepsilon$ du module de Verma $M^{h,h'}(n,g)$.
\item[2)] Soit $n \in \NN$. Le $U_{\varepsilon,h'}(\mathfrak{sl_2},g)$-module $L^{\varepsilon,h'}(n,g)$ est appelé la spécialisation à $Q = \varepsilon$ de la représentation indécomposable $L^{h,h'}(n,g)$.
\end{itemize}
\end{definition}

On vérifie immédiatement que les $U_{\varepsilon,h'}(\mathfrak{sl_2},g)$-modules $M^{\varepsilon,h'}(n,g)$ et $L^{\varepsilon,h'}(n,g)$ sont des modules de poids. \\

Supposons à présent que $n \in g \ZZ$ et explicitons les modules $M^{\varepsilon,h'}(n,g)$ et $L^{\varepsilon,h'}(n,g)$ :
\begin{eqnarray*}
H.m_j & = & (n-2j) \, m_j \, , \\
Q^H . m_j & = & (-1)^{n/g} \, \varepsilon^{-2j} \, m_j \, , \\
C. m_j & = & (n + 1)^2 \, m_j \, , \\
\big( Q^{\sqrt{C}} + Q^{-\sqrt{C}} \big) . m_j & = & (-1)^{n/g} \, (\varepsilon + \varepsilon^{-1} )  \, m_j \, , \\
X^- . m_j & = & m_{j+1} \, , \\
X^+ . m_j & = & \begin{cases}
\big[ j \big]_{\varepsilon \: \! T} \, \big[n -j+1 \big]_{\varepsilon} \, m_{j-1} \, , & \text{si } j \equiv 0 \ [g] \, , \\
\big[ j \big]_{\varepsilon} \, \big[n -j+1 \big]_{\varepsilon \: \! T} \, m_{j-1} \, , & \text{si } j \equiv 1 \ [g] \, , \\
\big[ j \big]_{\varepsilon} \, \big[n -j+1 \big]_{\varepsilon} \, m_{j-1} \, , & \text{sinon.}
\end{cases}
\end{eqnarray*}

Formellement on obtient l'action de $X^{\pm}$ sur $M^{\varepsilon,h'}(n,g)$ et $L^{\varepsilon,h'}(n,g)$ à partir de celle de $X^{\pm}$ sur les représentations de $U_h(\mathfrak{sl}_2)$ correspondantes, en remplaçant les boîtes quantiques $[j]_Q$ par les boîtes $[j]_{\varepsilon \: \! T}$ quand $g$ divise $j$, et par $[j]_{\varepsilon}$ sinon. \\
Dans le cas où $g = 2$ ou $3$, on retrouve alors formellement les actions de $X^{\pm}$ sur les représentations décrites dans \cite{hernandez}, avec $q=\varepsilon$ et $t = T$. \\

On suppose toujours $n \in g \ZZ$. \\
Notons ${}^L \! M^{h'}(n,g)$ et ${}^L \! L^{h'}(n,g)$ les sommes des espaces de poids divisibles par $g$ des $U_{\varepsilon,h'}(\mathfrak{sl_2},g)$-modules $ M^{\varepsilon,h'}(n,g)$ et $L^{\varepsilon,h'}(n,g)$ respectivement. \\
Posons $g' := g/2$ et $n' := 2 \: \!n/g$ si $g$ est pair, $g' := g$ et $n' := n/g$ si $g$ est impair. On a
$$ {}^L \! M^{h'}(n,g) \ \simeq \ \Big( \bigoplus_{j \in \NN} \CC \, m_{g' j} \Big) [[h']] \, , \quad \ {}^L \! L^{h'}(n,g) \ \simeq \ \Big( \bigoplus_{0 \leq j \leq n'} \CC \, m_{g' j} \Big) [[h']] \, . $$
D'après ce qui précède, d'une part $(X^{\pm})^g$, $H$, $C$, $Q^{\pm H}$ et $Q^{\sqrt{C}} + Q^{-\sqrt{C}}$ stabilisent ${}^L \! M^{h'}(n,g)$ et ${}^L \! L^{h'}(n,g)$, d'autre part $Q^{2H}$ agit par $1$ et $Q^{2 \sqrt{C}} + Q^{- 2\sqrt{C}}$ par $\varepsilon^2 + \varepsilon^{-2}$. \\

On note ${}^L \! M^{(h')}(n,g)$ et ${}^L \! L^{(h')}(n,g)$ les localisés par rapport à $h'$ des $\CC[[h']]$-modules ${}^L \! M^{h'}(n,g)$ et ${}^L \! L^{h'}(n,g)$. On a
$$ {}^L \! M^{(h')}(n,g) \ \simeq \ \Big( \bigoplus_{j \in \NN} \CC \, m_{g' j} \Big) ((h')) \quad \text{ et } \quad {}^L \! L^{(h')}(n,g) \simeq \Big( \bigoplus_{0 \leq j \leq n'} \CC \, m_{g' j} \Big) ((h')) \, . $$
Les actions de $U_{\varepsilon,h'}(\mathfrak{sl_2},g)$ induisent alors d'après la remarque \ref{rem_fond}, des actions de $U_{\! gh'}(\mathfrak{sl}_2)$ sur ${}^L \! M^{(h')}(n,g)$ et ${}^L \! L^{(h')}(n,g)$. \\

Grâce à la description précédente des modules $M^{\varepsilon,h'}(n,g)$ et $L^{\varepsilon,h'}(n,g)$, on vérifie que ${}^L \! X^{\pm}$ et ${}^L \! H$ stabilisent ${}^L \! M^{h'}(n,g)$ et ${}^L \! L^{h'}(n,g)$, considérés comme sous-$\CC[[h']]$-modules respectivement de ${}^L \! M^{(h')}(n,g)$ et ${}^L \! L^{(h')}(n,g)$. \\
${}^L \! M^{h'}(n,g)$ et ${}^L \! L^{h'}(n,g)$ sont alors des représentations de $U_{\! gh'}(\mathfrak{sl}_2)$

\begin{definition}
\begin{itemize}
\item[1)] Soit $n \in g \ZZ$. La représentation ${}^L \! M^{h'}(n,g)$ de $\Urh{\! gh'}$ est appellé la représentation Langlands $g$-duale du module de Verma $M^{h}(n)$ de $\Uh$.
\item[2)] Soit $n \in g \NN$. La représentation ${}^L \! L^{h'}(n,g)$ de $\Urh{\! gh'}$ est appellé la représentation Langlands $g$-duale de la représentation indécomposable $L^h(n)$ de $\Uh$.
\end{itemize}
\end{definition}

Dans la suite $\cmodh$ désignera la catégorie additive $\CC[[h]]$-linéaire des représentations de rang fini de $\Uh$ et $\groth$ son groupe de Grothendieck. \\
$\mathcal{C}^{gh'} \! (\mathfrak{sl}_2)$ désignera la catégorie additive $\CC[[h']]$-linéaire des représentation de rang fini de $\Urh{\! gh'}$ et $\text{Rep}^{\: \! gh'} \! (\mathfrak{sl}_2)$ son groupe de Grothendieck. \\

On dispose des foncteurs additifs
$$ \lim_{h' \to 0} : \ \cmodhh \, \longrightarrow \ \cmodh \, , \ \quad \lim_{h \to 0} : \ \cmodh \, \longrightarrow \ \cmod \, , $$
$$ \text{ et } \quad \lim_{h' \to 0} : \ \mathcal{C}^{gh'} \! (\mathfrak{sl}_2) \, \longrightarrow \ \cmod \, . $$

On note $P = \ZZ \, \! \omega$ le réseau des poids de $\mathfrak{sl}_2$ et $P^+ = \NN \: \! \omega \subset P$ le sous-ensemble des poids dominants. On définit une application $\Pi_g : P \to P$ par
$$ \Pi_g (n \: \! \omega) \ := \ \frac{n}{g} \: \! \omega \quad \text{si } g \, / \, n \, , \quad \Pi_g (n \: \! \omega) \ := 0 \quad \text{ sinon,} $$
qui s'étend linéairement en une application $\Pi_g : \ZZ[P] \to \ZZ[P]$ ($\Pi_g$ conserve l'addition mais pas le produit). \\

On note
$$ \chi \, : \ \grot \ \to \ \ZZ[P] \, , \quad \ \chi^h \, : \ \groth \ \to \ \ZZ[P] \, , $$
$$ \text{ et } \ \quad \chi^{g h'} \, : \ \text{Rep}^{gh'} \! (\mathfrak{sl}_2)\ \to \ \ZZ[P] $$
les morphismes de caractères (qui sont on le rappelle des morphismes d'anneau). \\

On peut, de la même façon, définir grâce au point 6 du théorème \ref{thm_rephh} un morphisme de caractère
$$ \chi_g^{h,h'} \, : \ \grothh \ \to \ \ZZ[P] \, . $$
La différence ici est que $\chi_g^{h,h'}$ est un morphisme seulement additif. \\

On vérifie sans difficulté que le diagramme suivant est commutatif :
$$ \xymatrix{
\grothh \ar[ddrr]_-{\chi_g^{h,h'}} \ \ar[rr]^-{\left[ \lim_{h' \to 0} \right]} && \ \groth \ar[dd]_-{\chi^h} \ \ar[rr]^-{\left[ \lim_{h \to 0} \right]} && \ \grot \ar[ddll]^-{\chi} \\
&&&& \\
&& \ZZ[P] &&
} $$

Les deux théorèmes qui suivent résument les constructions faites au début de cette section. Ils donnent également une description des représentations Langlands $g$-duales, qu'on obtient immédiatement d'après ce qui précède.

\begin{theorem} \label{thm_reprinter}
Soient $g \in \NN_{\geq 1}$ et $n \in g \: \! \NN$.
\begin{itemize}
\item[1)] La représentation $L^{h,h'}(n,g)$ de $\Uhh$ est une déformation selon $h'$ de la représentation indécomposable $L^h(n)$ de $\Uh$ :
$$ \lim_{h' \to 0} L^{h,h'}(n,g) \  = \ L^h(n) \, . $$

\item[2)] La spécialisation $L^{\varepsilon,h'}(n,g)$ à $Q = \varepsilon$ de $L^{h,h'}(n,g)$ est un $U_{\varepsilon,h'}(\mathfrak{sl}_2,g)$-module de poids.

\item[3)] $L^{\varepsilon,h'}(n,g)$ contient la représentation ${}^L \! L^{h'}(n,g)$ de $\Urh{gh'}$, Langlands $g$-duale de $L^h(n)$, comme somme des espaces de poids divisibles par $g$.

\item[4)] Les caractères de $L^h(n)$ et ${}^L \! L^{h'}(n,g)$ sont reliés par
$$ \left( \Pi_g \circ \chi^h \right) \left( L^h(n) \right) \ = \ \chi^{g h'} \left( {}^L \! L^{h'}(n,g) \right) . $$

\item[5)] Si $g$ est paire :
\begin{eqnarray*}
{}^L \! L^{h'}(n,g) & \simeq & L^{gh'} \left( \frac{n}{g} \right) \oplus L^{gh'} \left( \frac{n}{g} - 1 \right) \quad \text{ pour } n > 0 \, , \\
{}^L \! L^{h'}(0,g) & \simeq & L^{gh'}(0) .
\end{eqnarray*}
\item[6)] si $g$ est impaire :
$$ {}^L \! L^{h'}(n,g) \ \simeq \ L^{gh'} \left( \frac{n}{g} \right) . $$
\end{itemize}
\end{theorem}

\begin{definition}
La limite $h' \to 0$ de la représentation Langlands $g$-duale ($n \in g \NN$) ${}^L \! L^{h'}(n,g)$ est une représentation de $\mathfrak{sl}_2$ que l'on note ${}^L \! L(n,g)$ et qu'on appelle la représentation Langlands $g$-duale de la représentation irréductible $L(n)$ de $\mathfrak{sl}_2$.
\end{definition}

Au niveau des caractères on a :
$$ \left( \Pi_g \circ \chi \right) \left( L(n) \right) \ = \ \chi \left( {}^L \! L(n,g) \right) . $$

On peut illustrer le théorème \ref{thm_reprinter} et la remarque précédente par le diagramme suivant ($g \in \NN_{\geq 1}$, $n \in g \: \! \NN$) :
$$ \xymatrix{
&& L^{h,h'}(n,g) \ar[lld]_-{\lim_{h' \to 0 \ }} \ar[rrd]^-{\lim_{Q \to \varepsilon}} && \\
L^h(n) \ar@<-2.5pt>[d]_-{\lim_{h \to 0}} &&&& \quad {}^L \! L^{h'}(n,g) \ar@<2.5pt>[d]^-{\lim_{h' \to 0}} \ \supset \ L^{gh'}(n/g) \\
L(n) &&&& \quad {}^L \! L(n,g) \ \supset \ L(n/g)
} $$

En d'autres mots, la représentation irréductible $L(n)$ de $\mathfrak{sl}_2$, qui peut être déformée une première fois selon $h$ en une représentation $L^h(n)$ de $\Uh$, peut une seconde fois être déformée selon $h'$ en une représentation $L^{h,h'}(n,g)$ de $\Uhh$, les rangs des espaces de poids restant invariants sous chacune des déformations. \\
Par ailleurs, la spécialisation à $Q = \varepsilon$ de cette double déformation $L^{h,h'}(n,g)$ contient la représentation ${}^L \! L^{h'}(n,g)$ de $\Urh{gh'}$, Langlands $g$-duale de $L^h(n)$. ${}^L \! L^{h'}(n,g)$ est en outre la déformation selon $h'$ de la représentation ${}^L \! L(n,g)$ de $\mathfrak{sl}_2$, Langlands $g$-duale de la représentation $L(n)$ de $\mathfrak{sl}_2$.

\begin{theorem} \label{thm_reprinter2}
Soit $n \in g \: \! \ZZ$.
\begin{itemize}
\item[1)] Le module de Verma $M^{h,h'}(n,g)$ de $\Uhh$ est une déformation selon $h'$ du module de Verma $M^h(n)$ de $\Uh$ :
$$ \lim_{h' \to 0} M^{h,h'}(n,g) \  = \ M^h(n) \, . $$
\item[2)] La spécialisation $M^{\varepsilon,h'}(n,g)$ à $Q = \varepsilon$ de $M^{h,h'}(n,g)$ est un $U_{\varepsilon,h'}(\mathfrak{sl}_2,g)$-module de poids.
\item[3)] $M^{\varepsilon,h'}(n,g)$ contient la représentation ${}^L \! M^{h'}(n,g)$ de $\Urh{gh'}$, Langlands $g$-duale de $M^h(n)$, comme somme des espaces de poids divisibles par $g$.
\item[4)] Si $g$ est paire :
$$ {}^L \! M^{h'}(n,g) \ \simeq \ M^{gh'} \left( \frac{n}{g} \right) \oplus M^{gh'} \left( \frac{n}{g} - 1 \right) . $$
\item[5)] si $g$ est impaire :
$$ {}^L \! M^{h'}(n,g) \ \simeq \ M^{gh'} \left( \frac{n}{g} \right) . $$
\end{itemize}
\end{theorem}

\begin{definition}
Soit $n \in g \ZZ$. La limite quand $h'$ tend vers $0$ de la représentation Langlands $g$-duale ${}^L \! M^{h'}(n,g)$ est une représentation de $\mathfrak{sl}_2$ que l'on note ${}^L \! M(n,g)$ et qu'on appelle la représentation Langlands $g$-duale du module de Verma $M(n)$ de $U(\mathfrak{sl}_2)$.
\end{definition}

On peut illustrer le théorème \ref{thm_reprinter2} et la remarque précédente par le diagramme suivant ($g \in \NN_{\geq 1}$, $n \in g \: \! \ZZ$) :
$$ \xymatrix{
&& M^{h,h'}(n,g) \ar[lld]_-{\lim_{h' \to 0 \ }} \ar[rrd]^-{\lim_{Q \to \varepsilon}} && \\
M^h(n) \ar@<-2.5pt>[d]_-{\lim_{h \to 0}} &&&& \quad {}^L \! M^{h'}(n,g) \ar@<2.5pt>[d]^-{\lim_{h' \to 0}} \ \supset \ M^{gh'}(n/g) \\
\quad M(n) &&&& {}^L \! M(n,g) \ \supset \ M(n/g)
} $$

En d'autres mots, le module de Verma $M(n)$ de $\mathfrak{sl}_2$, qui peut être déformée une première fois selon $h$ en une représentation $M^h(n)$ de $\Uh$, peut une seconde fois être déformée selon $h'$ en une représentation $M^{h,h'}(n,g)$ de $\Uhh$. \\
Par ailleurs, la spécialisation à $Q = \varepsilon$ de cette double déformation $M^{h,h'}(n,g)$ contient la représentation ${}^L \! M^{h'}(n,g)$ de $\Urh{gh'}$, Langlands $g$-duale de $M^h(n)$. ${}^L M^{h'}(n,g)$ est la déformation selon $h'$ de la représentation ${}^L \! M(n,g)$ de $\mathfrak{sl}_2$, Langlands $g$-duale du module de Verma $M(n)$ de $\mathfrak{sl}_2$. \\

L'étude précédente des représentations $L^{h,h'}(n,g)$ de $U_{\: \! h,h'}(\mathfrak{sl}_2,g)$ et de leurs spécialisations à $Q = \varepsilon$ permet de terminer la preuve des propriétés d'interpolation de $\Uhh$ entre les groupes quantiques $\Uh$ et $\Urh{\! gh'}$.

\begin{theorem} \label{thm_uhhinterpolation}
\begin{itemize}
\item[1)] $\Uhh$ est une déformation formelle de $\Uh$ selon le paramètre $h'$. \\
En particulier, $\Uhh / (h' = 0)$ est isomorphe, en tant que $\CC[[h]]$-algèbre, à $\Uh$, via l'identification $\tilde{X}^{\pm} = X^{\pm}$ et $\tilde{H} = H$.

\item[2)] La sous-$\CC[[h']]$-algèbre $\langle {}^L \! X^{\pm}, {}^L H \rangle$ de
$$ \Bigg( \, U_{\! \: \mathcal{A}, h'} (\mathfrak{sl}_2, g) \ / \ \overline{ \big( Q = \varepsilon, \, Q^{2H} = 1, \, Q^{2 \sqrt{C}} + Q^{-2 \sqrt{C}} = \varepsilon^{2} + \varepsilon^{-2} \big) }^{ \, h'} \, \Bigg)_{\! \! h'} \, , $$
 topologiquement engendrée par ${}^L \! X^{\pm}$ et ${}^L H$, est isomorphe à $\Urh{\! gh'}$.
\end{itemize}
\end{theorem}

\begin{proof}
Le premier point a déjà été traité : voir le théorème \ref{thm_uhhdefor}. \\
Quant au second, d'après le théorème \ref{thm_reprinter}, quelque soit $m \in \NN$, l'action de $\Urh{\! gh'}$ sur la représentation $L^{gh'} (m)$ se factorise à travers le morphisme surjectif
$$ U_{\! g h'}(\mathfrak{sl}_2) \ \longrightarrow \ \langle {}^L \! X^{\pm}, {}^L \! H \rangle $$
donné dans la remarque \ref{rem_fond}. Par suite ce morphisme est injectif. En effet un élément de $U_{\! g h'}(\mathfrak{sl}_2)$ appartenant au noyau de la surjection précédente agit alors par zéro sur toutes ses représentations de dimension finie, et par suite est nul. Ce dernier fait repose sur l'existence de la base PBW de $U_{\! g h'}(\mathfrak{sl}_2)$ et sur la description de ses représentations indécomposables de rang finie (on pourra consulter \cite[5.11]{jantzen} pour un résultat analogue dans le cas de $U_q(\Glie)$, où $\Glie$ est une algèbre de Lie semi-simple de dimension finie; notons toutefois que la démonstration dans le de $U_q(\mathfrak{sl}_2)$ ou $U_h(\mathfrak{sl}_2)$ est plus simple).
\end{proof}

On peut illustrer le théorème précédent par le diagramme suivant :
$$ \xymatrix{
&& \Uhh \ar[lld]_-{\lim_{h' \to 0 \ }} \ar[rrd]^-{\lim_{Q \to \varepsilon}} && \\
\Uh \ar@<-2.5pt>[d]_-{\lim_{h \to 0}} &&&& \Urh{gh'} \ar@<2.5pt>[d]^-{\lim_{h' \to 0}} \\
\Uc &&&& \Uc
} $$

En d'autres mots, l'algèbre enveloppante $\Uc$, qui admet le groupe quantique $\Uh$ comme première déformation selon $h$, est déformée une seconde fois selon $h'$ pour chaque $g \in \NN_{\geq 1} $ : on obtient alors les groupes quantiques d'interpolation de Langlands. \\
Ceux-ci, quand ils sont spécialisés à $Q = \exp h = \varepsilon$, permettent de retrouver les groupes quantiques $\Urh{\! gh'}$, déformations selon $h'$ de $\Uc$.

\section*{Pistes de Recherche}
Les groupes quantiques d'interpolation de Langlands de rang $1$ constituent la brique élémentaire, à partir de laquelle seront construits, dans un prochain article, les groupes quantiques d'interpolation de Langlands $\widetilde{U}_{h,h'}(\Glie,g)$ associés à une algèbre de Kac-Moody (symétrisable). \\
Une question naturelle et importante est : existe-t-il une double déformation des relations de Serre, de telle sorte qu'il soit possible de construire des groupes quantiques d'interpolation $U_{h,h'}(\Glie,g)$ qui double déforment l'algèbre enveloppante $U(\Glie)$? Cette question semble très liée à celle de l'existence d'une structure d'algèbre de Hopf sur $\Uhh$ et par suite sur $\widetilde{U}_{h,h'}(\Glie,g)$. Plusieurs indices semblent indiquer qu'il pourrait en effet exister un coproduit sur $\Uhh$ (ou du moins une structure tensorielle sur sa catégorie des représentations) qui interpole les coproduits de $\Uh$ et $\Urh{gh}$. \\
Une nouvelle direction de recherche apparaît alors : trouver des $R$-matrices universelles pour les groupes quantiques d'interpolation, afin de produire des invariants de noeuds, qui en un sens contiendraient simultanément les invariants de noeuds fournis par les deux groupes quantiques interpolés. \\

Signalons que l'existence d'une double déformation des représentations $L(\lambda)$ possédant la propriété d'interpolation \eqref{eq_reprinter} est liée à l'étude dans \cite{hernandez} de la dualité de Langlands au niveau des cristaux. On peut se demander alors si il serait possible de développer une théorie des bases cristallines pour les groupes quantiques d'interpolation de Langlands, dont une des propriétés serait en un certain sens d'interpoler les théories déjà existantes pour $\Glie$ et ${}^L \Glie$ : voir par exemple \cite[conjecture 3]{hernandez}. \\

Dernière voie envisageable : dans le cas où $\Glie$ est une algèbre de Kac-Moody affine, il serait intéressant de se demander si les groupes quantiques d'interpolation associés à $\Glie$ permettent d'expliquer la dualité de Langlands au niveau des $q$-caractères, déjà démontrée pour les modules de Kirillov-Reschetikhin dans \cite{hernandez2}.

\acks Je voudrais remercier mon directeur de thèse, David Hernandez, qui m'a proposé d'étudier ce problème et soutenu dans mes recherches. J'aimerais aussi remercier Marc Rosso et Éric Vasserot d'avoir pris le temps de répondre à quelques-unes de mes questions. Enfin je suis reconnaissant à tous ceux, et en particulier à Yohan Brunebarbe, Xin Fang, Dragos Fratila, Mathieu Mansuy et Sarah Scherotzke, qui, grâce à des échanges enrichissants, m'ont permis de mieux avancer.

\nocite{*}
\bibliographystyle{alpha-fr}
\bibliography{grp_q_interpolation_rg1_imrn}

\end{document}